\documentclass[12pt,reqno,oneside]{amsart}
\usepackage{amscd,amssymb,amsmath,amsthm}
\usepackage[dvips]{graphicx}
\usepackage{mathrsfs}

\newtheorem{thm}{Theorem}[section]
\newtheorem{lem}[thm]{Lemma}
\newtheorem{prop}[thm]{Proposition}
\newtheorem{defn}[thm]{Definition}
\newtheorem{cor}[thm]{Corollary}
\newtheorem{rem}[thm]{\textit{\textrm{Remark}}}
\newtheorem{notation}[thm]{Notation}
\newtheorem{ex}[thm]{\textit{\textrm{Example}}}
\newtheorem{note}{\textit{\textrm{Note}}}
\numberwithin{equation}{section}

\usepackage{hyperref}

\numberwithin{equation}{section}
\newcommand{\NI}{\noindent}

\newcommand\HUGE{\@setfontsize\Huge{38}{47}}
\begin{document}
\title[Eigenvalues of Gram Matrices of a Class of Diagram Algebras] {Eigenvalues of Gram Matrices of a Class of Diagram Algebras}

\  \author{N.\,Karimilla Bi}  \author{M.\,Parvathi$^\dag$}
 \maketitle{\small{

\begin{center}
Ramanujan Institute for Advanced Study in Mathematics, \,\\
University  of  Madras,  \\
Chepauk, Chennai -600 005, Tamilnadu, India.\\
{\bf {$^\dag$ E-Mail: sparvathi@hotmail.com}}
\end{center}}

\begin{abstract}
In this paper, we introduce symmetric diagram matrices $A^{s+r, s}$ of size ${_{(s+r)}}C_s$ whose entries are $\{x_i\}_{1  \leq  i  \leq  \text{min}\{s, r\}}$. We compute the eigenvalues of symmetric diagram matrices using elementary row and column operations inductively. As a byproduct, we obtain the eigenvalues of Gram matrices of a larger class of diagram algebras like the  signed partition algebras,  algebra of $\mathbb{Z}_2$ relations and partition algebras.
\end{abstract}

\quad\quad \textbf{keywords:} Eigenvalues, Symmetric diagram matrices,  Gram matrices, partition algebras, 

\quad\quad signed partition algebras and algebra of $\mathbb{Z}_2$-relations.


\section{\textbf{Introduction}}

The study of the structure of finite dimensional
algebras has gained importance in recent times for it may be possible to find presumably new
examples of subfactors of a hyper finite $\Pi_1$-factor along the
lines of \cite{W}.

The partition algebras are introduced by V. Jones  in \cite{J} and Martin in \cite{PM2} and they are studied intensively by Martin in \cite{PH, PM1, PM2, PM3, PM4, PM5}. The cellularity of many subalgebras of partition algebras like Tanabe algebras, Party algebras, algebra of $\mathbb{Z}_2$-relations and signed partition algebras are studied in detail in \cite{T}, \cite{KS}, \cite{VSS} and \cite{K, SP} respectively.

The question of semisimplicity of these algebras reduces to the nondegeneracy of Gram matrices of such algebras. In this connection, the Gram matrices of partition algebras, algebra of $\mathbb{Z}_2$-relations and signed partition algebras are studied in \cite{KP} which are realized as direct sum of block sub matrices. The diagonal entry of the block submatrix is a product of $r_1$ quadratic polynomials and $r_2$ linear polynomials in the case of algebra of $\mathbb{Z}_2$-relations and signed partition algebras and it is a product $r$ linear polynomials in the case of partition algebras.

 In this paper, we introduce symmetric diagram matrices $A^{s+r, s}$ of size ${_{(s+r)}}C_s$ whose entries are $\{x_i\}_{1  \leq  i  \leq  \text{min}\{s, r\}}$ based on the diagrams. We compute the eigenvalues of symmetric diagram matrices using elementary row and column operations by induction.

  As a byproduct, we compute the eigenvalues of Gram matrices of a larger class of diagram algebras like the signed partition algebras, partition algebras and  algebra of $\mathbb{Z}_2$ relations where the Gram matrix of any diagram algebra is a direct product of symmetric diagram matrices.
\section{\textbf{Preliminaries}}

\subsection{\textbf{Partition Algebras}}
\textbf{\\}
We recall the definitions in \cite{HA} required in this paper. Let $X$ be any set and let $R_X$ denote the set of all equivalence relations on $X.$ For $k \in \mathbb{N},$ let $\underline{k} = \{1,2, \cdots, k\}, \underline{k}' = \{1', 2', \cdots, k'\}.$ Let $X = \underline{k} \cup \underline{k}'$ and $R_{k \cup k'}$ be the set of all partitions  of $\{\underline{k} \cup \underline{k}'\}$ or equivalence relation on $\underline{k} \cup \underline{k}'.$ For $d \in R_{k \cup k'},$ the elements of $d$ are called as connected components.

Any $d \in R_{k \cup k'}$ can be represented as a simple graph on two rows of $k$-vertices, one above the other with $k$ vertices in the top row labeled $1,2,\cdots, k$ left to right and $k$ vertices in the bottom row labeled $1', 2', \cdots, k'$ left to right  with vertex $i$ and vertex $j$ connected by a path if $i$ and $j$ are in the same block of the set partition $d.$ The graph representing $d$ is called $k$-partition diagram and it is not unique. Two $k$-partition diagrams are said to be equivalent if they give rise to the same set partition of $2k$-vertices.

A connected component of $d$  containing an element of $\{1, 2, \cdots, k\}$ and element of $\{1', 2', \cdots, k'\}$ is  called  {\it through class} and a connected component containing an element of either $\{1, 2, \cdots, k\}$ or $\{1', 2', \cdots, k'\}$ is called {\it horizontal edge}.

\NI Number of through classes in $d$ is called {\it propagating number} and it is denoted by $\sharp^p(d).$

We shall define multiplication of two $k$-partition diagrams $d'$ and $d''$ as follows:

\begin{itemize}
  \item Place $d'$ above $d''$ and identify the bottom dots of $d'$ with the top dots of $d''.$
  \item $d' \circ d''$ is the resultant diagram obtained by using only the top row of $d'$ and bottom row of $d'',$ replace each connected component which lives entirely in the middle row by the variable $x$. i.e., $d' \circ d'' = x^{\lambda} d'''$ where $\lambda$ is the number of connected components that lie entirely in the middle row.
\end{itemize}

\NI This product is associative and is independent of the graph we choose to represent the $k$-partition diagram.

Let $\mathbb{C}(x)$ be the field of rational functions with complex coefficients in the variable $x.$ The partition algebra $A_k(x)$ is defined to be the $\mathbb{C}(x)$-span of the $k$-partition diagrams, which is an associative algebra with identity $1$ where

\centerline{\includegraphics[height=1.5cm, width=4cm]{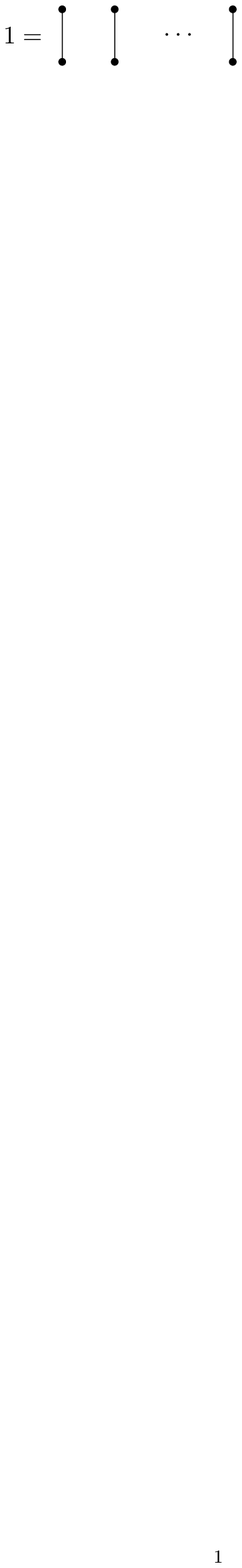}}

By convention $A_0(x) = \mathbb{C}(x).$

For $1 \leq i \leq k-1$ and $1 \leq j \leq k,$ the following are the generators of the partition algebras.
\begin{center}
\includegraphics[height=3.5cm, width=5cm]{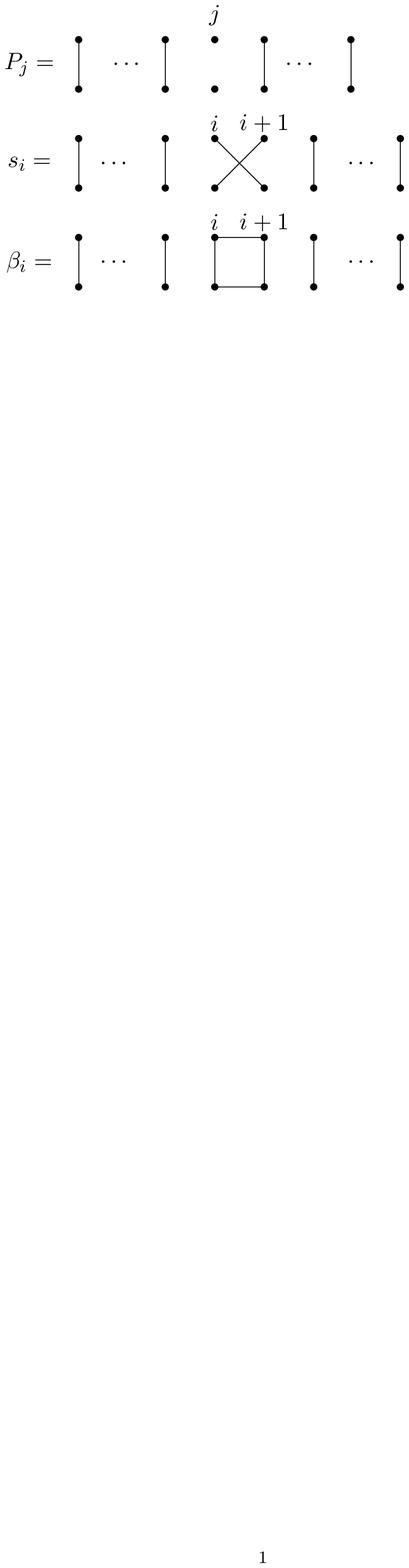}

\end{center}

\NI The above generators satisfy the relations given in Theorem 1.11 of \cite{HA}.

\subsection{The algebra of $\mathbb{Z}_2$-relations:}
\begin{defn}\label{D2.1} (\cite{VSS})
Let the group $\mathbb{Z}_2$ act on the set $X.$ Then the action of $\mathbb{Z}_2$ on $X$ can be extended to an action of $\mathbb{Z}_2$ on $R_X,$ where $R_X$ denote the set of all equivalence relations on $X$, given by

\centerline{$g.d = \{(gp, gq) \ | \ (p, q) \in d\}$}
\NI where $d \in R_X$ and $g \in \mathbb{Z}_2.$ (It is easy to see that the relation $g.d$ is again an equivalence relation).

An equivalence relation $d$ on $X$ is said to be a $\mathbb{Z}_2$-stable equivalence relation if $p \sim q$ in $d$ implies that $gp \sim gq$ in $d$ for all $g$ in $\mathbb{Z}_2.$ We denote $\underline{k}$ for the set $\{1, 2, \cdots, k\}.$ We shall only consider the case when $\mathbb{Z}_2$ acts freely on $X$; let $X = \underline{k} \times \mathbb{Z}_2$ and the action is defined by $g.(i, x) = (i, gx)$ for all $1 \leq i \leq k.$

Let $R_k^{\mathbb{Z}_2}$ be the set of all $\mathbb{Z}_2$-stable equivalence relations on $X.$
\end{defn}

\begin{notation}\label{N2.2}(\cite{VSS})
$R_k^{\mathbb{Z}_2}$ denotes the set of all $\mathbb{Z}_2$-stable equivalence relation on $\{1,2, \cdots k\} \times \mathbb{Z}_2.$

Each $d \in R_k^{\mathbb{Z}_2}$ can be represented as a simple graph on a row of $2k$ vertices.
\begin{itemize}
  \item The vertices $(1, e), (1, g), \cdots, (k, e), (k, g)$ are arranged from left to right in a single row.
  \item If $(i, g) \sim (j, g') \in R_k^{\mathbb{Z}_2}$ then $(i, g), (j, g')$ is joined by a line for all $g, g' \in \mathbb{Z}_2.$
\end{itemize}
We say that the two graphs are equivalent if they give rise to the same set partition of the $2k$ vertices $\{(1, e), (1, g), \cdots, (k, e), (k, g)\}.$

\NI We may regard each element $d$ in $R_{k \cup k'}^{\mathbb{Z}_2}$ as a $2k$-partition diagram by arranging the $4k$ vertices $(i, g), i \in \underline{k} \cup \underline{k}', g \in \mathbb{Z}_2$ of $d$ in two rows in such a way that $(i, g)\left( (i', g)\right)$ is in the top(bottom) row of $d$ if $1 \leq i \leq k(1' \leq i' \leq k')$ for all $g \in \mathbb{Z}_2$  and if $(i, g) \sim (j, g')$ then $(i, g), (j, g')$ is joined by a line for all $g, g' \in \mathbb{Z}_2.$

The diagrams $d^+$ and $d^-$ are obtained from the diagram $d$ by restricting the vertex set to $\{(1, e), (1, g), \\ \cdots, (k, e), (k, g)\}$ and $\{(1', e), (1', g), \cdots (k', e), (k', g)\}$ respectively. The diagrams $d^+$ and $d^-$ are also $\mathbb{Z}_2$-stable equivalence relation and $d^+, d^- \in R_k^{\mathbb{Z}_2}.$
\end{notation}

\begin{defn}\label{D2.3}(\cite{VSS})
Let $d \in R_{k \cup k'}^{\mathbb{Z}_2}.$ Then the equation

\centerline{$R^d = \{(i, j) \ | \ \text{ there exists } g, h \in \mathbb{Z}_2 \text{ such that } ((i, g), (i, h)) \in d\}$}
\NI defines an equivalence relation on $\underline{k} \cup \underline{k}'.$
\end{defn}

\begin{rem}\label{R2.4} (\cite{VSS})
 For $d \in R_{k \cup k'}^{\mathbb{Z}_2}$ and for every $\mathbb{Z}_2$-stable equivalence class or a connected component $C$ in $R^d$ there exists a unique subgroup denoted by $H_C^d$ where
\begin{enumerate}
  \item[(i)]  $H_C^d = \{e\}$ if $(i, e) \not \sim (i, g) \ \forall \ i \in C, C$ is called an $\{e\}$-class or $\{e\}$-component and the $\{e\}$-component $C$ will always occur as a pair and
  \item[(ii)] $H_C^d =\mathbb{Z}_2$ if $(i, e) \sim (i, g) \ \forall \ i \in C, C$ is called a $\mathbb{Z}_2$-class or $\mathbb{Z}_2$-component and the number of vertices in the $\mathbb{Z}_2$-component $C$ will always be even.
\end{enumerate}

\end{rem}

\begin{prop}\label{P2.5} (\cite{VSS})
The linear span of $R_{k \cup k'}^{\mathbb{Z}_2}$ is a subalgebra of $A_{2k}(x)$. We denote this subalgebra by $A_k^{\mathbb{Z}_2}(x),$ called the algebra of $\mathbb{Z}_2$-relations.
\end{prop}

\begin{defn}\label{D2.6}(\cite{VSS})
For $0 \leq 2s_1+s_2 \leq 2k,$ define $I^{2k}_{2s_1+s_2}$ as follows:

\centerline{ $I^{2k}_{2s_1+s_2} = \left\{ d \in R_{k \cup k'}^{\mathbb{Z}_2} \ | \ \sharp^p(d) = 2s_1+s_2 \right\}$}

i.e., $d$ has $s_1$ number of pairs of $\{e\}$-through classes and $s_2$ number of $\mathbb{Z}_2$-through classes.

 It is clear that $R_{k \cup k'}^{\mathbb{Z}_2} = \underset{\substack{0 \leq s_1, s_2 \leq 2k\\ 0 \leq 2s_1+s_2 \leq 2k}}{\cup} I^{2k}_{2s_1+s_2}.$
\end{defn}

\subsection{Signed Partition Algebras:}
\begin{defn}\label{D2.7} (\cite{SP}, \textbf{Definition 3.1.1})
Let the signed partition algebra $\overrightarrow{A}_k^{\mathbb{Z}_2}(x)$ be the subalgebra of $A_{2k}(x)$ generated by
\begin{center}
\includegraphics[height=6cm, width=8cm]{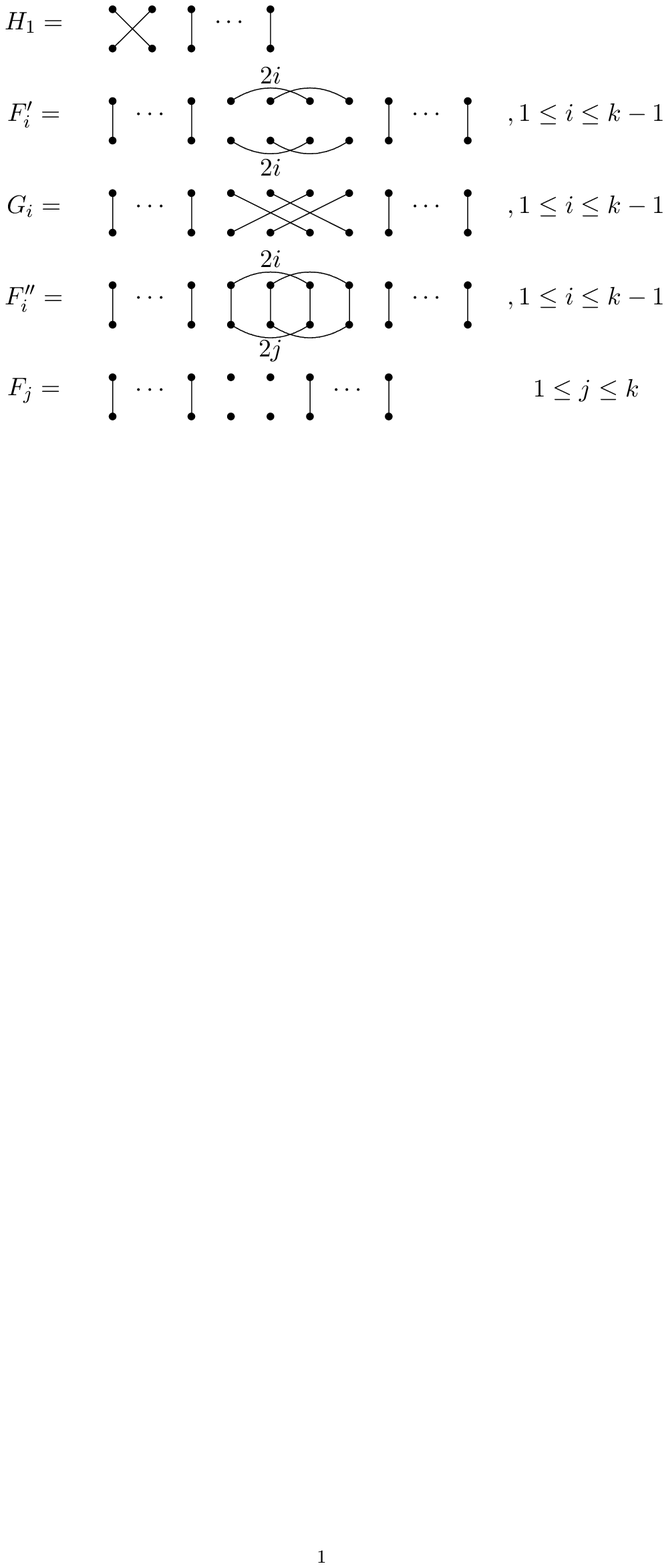}
\end{center}
\vspace{-0.5cm}
The subalgebra of the signed partition algebra generated by $F'_i, G_i, F''_i, F_j, 1 \leq i \leq k -1, 1 \leq j \leq k$ is isomorphic on to the partition algebra $A_{2k}(x^2).$
\end{defn}
\begin{defn}\label{D2.8}(\cite{SP}, \textbf{Definition 3.1.2})
Let $d \in R_{k \cup k'}^{\mathbb{Z}_2}$. For $0 \leq r \leq 2k -1, 0 \leq s_1, s_2 \leq k - 1,$

\NI $\widetilde{I}^{2k}_{2s_1+s_2} = \left\{ d \in I^{2k}_{2s_1+s_2} \ | \ s_1+s_2 + H_e(d^+) + H_{\mathbb{Z}_2}(d^+) \leq k -1 \text{ and } s_1+s_2 + H_e(d^-) + H_{\mathbb{Z}_2}(d^-) \leq k -1 \right\}$ where
\begin{enumerate}
  \item[(i)] $s_1 = \natural \{C : C$ is a through class of $R^d$ such that $H_C^d = \{e\}\},$
  \item[(ii)] $s_2 = \natural \{C : C$ is a through class of $R^d$ such that $H_C^d = \mathbb{Z}_2 \},$
  \item[(iii)] $H_e(d^+) \left( H_e(d^-)\right)$ is the number of $\{e\}$ horizontal edges $C$ in the top(bottom) row of $R^d$ such that $H_C^d = \{e\}$ and $|C| \geq 2,$
  \item[(iv)] $H_{\mathbb{Z}_2}(d^+) \left( H_{\mathbb{Z}_2}(d^-)\right)$ is the number of $\mathbb{Z}_2$ horizontal edges $C$ in the top(bottom) row of $R^d$ such that $H_C^d = \mathbb{Z}_2$
  \item[(v)] $\sharp^p \left( R^d \right) = s_1 + s_2.$
\end{enumerate}
\NI The linear span of $\underset{\substack{0 \leq s_1 \leq k , 0 \leq s_2 \leq k-1 \\ 0 \leq s_1+s_2 \leq k-1}}{\cup}$ is nothing but the signed partition algebra $\overrightarrow{A}_k^{\mathbb{Z}_2}.$
\end{defn}

\begin{rem}\label{R2.9}
The algebra generated by $R^d$ where $d \in \{F'_i, G_i, F''_i, F_j\}_{\substack{1 \leq i \leq k-1 \\ \hspace{-0.4cm}1 \leq j \leq k}}$ is isomorphic to the partition algebra $A_k(x).$

\NI Also, let $I_s^k$ be the set of all $k$-partition diagrams $R^d$ in $A_k(x)$ such that $\sharp^p \left(R^d \right) = s$ where $d \in I^{2k}_{2s_1+0},$ which is contained in $ A_{2k}(x^2).$
\end{rem}

\begin{defn}\label{D2.10}
\textbf{\\}
\begin{enumerate}
  \item[(a)] Let $d \in R_{k \cup k'}^{\mathbb{Z}_2}$ with $s_1$ number of pairs of $\{e\}$-through classes, $s_2$ number of $\mathbb{Z}_2$-through classes, $r_1$ number of pairs of $\{e\}$-horizontal edges and $r_2$ number of $\mathbb{Z}_2$-horizontal edges.

      Let the number of vertices in the top row of $s_1$ number of pairs of $ \{e\}$-through classes($s_2$ number of $\mathbb{Z}_2$-through classes) be denoted by $\lambda_{1j}, 1 \leq j \leq s_1 \left( \lambda_{2j}, 1 \leq j \leq s_2\right)$ and let the number of vertices in the top row of $r_1$ number of pairs of $\{e\}$-horizontal edges ($r_2$ number of horizontal edges) be denoted by $\lambda_{3j}, 1 \leq j \leq r_1 \left( \lambda_{4j}, 1 \leq j \leq r_2\right)$ such that

      \centerline{$\phi(d^+) = \left( \lambda_1, \lambda_2, \lambda_3, \lambda_4\right)$}

       \NI where $\phi: R_{k \cup k'}^{\mathbb{Z}_2} \rightarrow P(2k), \lambda_1 = \left( \lambda_{11}, \lambda_{12}, \cdots, \lambda_{1s_1}\right), \lambda_2 = \left( \lambda_{21}, \lambda_{22}, \cdots, \lambda_{2s_2}\right), \\ \lambda_3 = \left( \lambda_{31}, \lambda_{32}, \cdots, \lambda_{3r_1}\right), \lambda_4 = \left( \lambda_{41}, \lambda_{42}, \cdots, \lambda_{4r_2}\right)$ with $|\lambda_1|+|\lambda_2|+|\lambda_3|+|\lambda_4| = 2k$ and $d^+$ is obtained from $d$ by restricting the vertex set to  $\{(1, e), (1, g), \cdots, (k, e), (k, g)\}.$
  \item[(b)]Let $d \in R_{k \cup k'}$ with $s$ number of through classes and $r$ number of horizontal edges.

      Let the number of vertices in the top row of $s$ number of through classes be denoted by $\lambda_{1j}, 1 \leq j \leq s $ and let the number of vertices in the top row of $r$ number of horizontal edges  be denoted by $\lambda_{2j}, 1 \leq j \leq r $ such that

      \centerline{$\phi'(d^+) = \left( \lambda_1, \lambda_2\right)$}

       \NI where $\phi' : R_{k \cup k'} \rightarrow P(k), \lambda_1 = \left( \lambda_{11}, \lambda_{12}, \cdots, \lambda_{1s}\right), \lambda_2 = \left( \lambda_{21}, \lambda_{22}, \cdots, \lambda_{2r}\right)$ with $|\lambda_1| + |\lambda_2| = k$ and $d^+$ is obtained from $d$ by restricting the vertex set to  $\{1, 2, \cdots, k\}.$
  \end{enumerate}
\end{defn}

\subsection{Gram Matrices of algebra of $\mathbb{Z}_2$-relations, signed partition algebras and partition algebras}

\begin{defn} \label{D2.11}\textbf{(\cite{KP}, Definition 3.1)}
Define,
\begin{enumerate}
  \item[(a)] $\Omega_{s_1, s_2}^{r_1, r_2} = \Big\{ \left[ \lambda_1^2 \right]^1 \left[ 2\lambda_2 \right]^2 \left[ \lambda_3^2 \right]^3 \left[ 2 \lambda_4 \right]^4 \Big| \lambda_1 \vdash k_1, \lambda_2 \vdash k_2, \lambda_3 \vdash k_3, \lambda_4 \vdash k_4  \text{ with } \lambda_1 \in \mathbb{P}(k_1, s_1),  $

$ \hspace{2cm}  \lambda_2 \in \mathbb{P}(k_2, s_2), \lambda_3 \in \mathbb{P}(k_3, r_1), \lambda_4 \in \mathbb{P}(k_4, r_2) \text{ such that } k_1 + k_2 + k_3 + k_4 = k \Big\}$

\NI where $\lambda_1^2 = \left( \lambda_{11}^2, \lambda_{12}^2, \cdots, \lambda_{1 s_1}^2 \right), 2 \lambda_2 = \left( 2 \lambda_{21}, 2 \lambda_{22}, \cdots, 2 \lambda_{2 s_2} \right), \lambda_3^2 = \left( \lambda_{31}^2, \lambda_{32}^2, \cdots, \lambda_{3 r_1}^2 \right)$ \\ and $2 \lambda_4 = \left( 2 \lambda_{41}, 2 \lambda_{42}, \cdots, 2 \lambda_{4 r_2}\right).$

  \item[(b)] $\Omega^{r}_{s} = \{ [\lambda_1]^1 [\lambda_2]^2 \ | \ \lambda_1 \in \mathbb{P}(k_1, s), \lambda_2 \in \mathbb{P}(k_2, r) \text{ such that } k_1 + k_2 = k\}.$
\end{enumerate}

\end{defn}

\begin{notation}\label{N2.12} \textbf{(\cite{KP}, Notation 3.5)}

\begin{enumerate}
            \item[(a)]  For $ 0 \leq r_1 \leq k-s_1 - s_2, 0 \leq r_2 < k -s_1-s_2, $ $ 0 \leq s_1 \leq k$ and $0 \leq s_2 \leq k, $ put $ J^{2k}_{2s_1+s_2} = \underset{\substack{0 \leq r_1 \leq k-s_1-s_2 \\ 0 \leq r_2 \leq k-s_1-s_2}}{\cup} \ \  \mathbb{J}_{2s_1+ s_2}^{2r_1 + r_2}$

                \NI where $\mathbb{J}_{2s_1 + s_2}^{2r_1+ r_2} = \Big\{ d \in I^{2k}_{2s_1 + s_2} \ \Big| \  d = U^{(d, P)}_{(d, P)},  d^+ $ and $d^-$ are the same, $\sharp^p\left(U^{(d, P)}_{(d, P)}\right) = 2s_1 + s_2, U^{(d, P)}_{(d, P)}$ has $r_1$ number of pairs of $\{e\}$-horizontal edges and $r_2$ number of $\mathbb{Z}_2$-horizontal edges, $(d, P)$ denotes the top row of $d, P$ is a subset of $d$ such that $|P| = 2s_1+s_2 \Big\}.$

 Also,  $\left| \mathbb{J}_{2s_1 + s_2}^{2 r_1 + r_2} \right|  =  f^{2r_1+r_2}_{2s_1+s_2}$ and $\left| J^{2k}_{2s_1+s_2}\right| = f_{2s_1+s_2}.$

\item[(b)]  For $ 0 \leq r_1 \leq k-s_1 - s_2, 0 \leq r_2 < k -s_1-s_2-1, $ $ 0 \leq s_1 \leq k$ and $0 \leq s_2 \leq k - 1,  s_1+s_2+r_1+r_2 \leq k - 1$ and if $s_1+s_2+r_1+r_2 = k$ then $s_1 = k$ or $r_1 \neq 0$

    \NI Put, $ \widetilde{J}^{2k}_{2s_1+s_2} = \underset{\substack{0 \leq r_1 \leq k-s_1-s_2 \\ 0 \leq r_2 < k-s_1-s_2}}{\cup} \ \  \widetilde{\mathbb{J}}_{2s_1+ s_2}^{2r_1 + r_2}$

\NI where $\widetilde{\mathbb{J}}_{2s_1 + s_2}^{2r_1+ r_2} = \Big\{ \widetilde{d} \in \widetilde{I}^{2k}_{2s_1 + s_2} \ \Big| \  \widetilde{d} = U^{(\widetilde{d}, \widetilde{P})}_{(\widetilde{d}, \widetilde{P})},  \widetilde{d}^+ $ and $\widetilde{d}^-$ are the same, $\sharp^p\left(U^{(\widetilde{d}, \widetilde{P})}_{(\widetilde{d}, \widetilde{P})}\right) = 2s_1 + s_2,  U^{(\widetilde{d}, \widetilde{P})}_{(\widetilde{d}, \widetilde{P})}$ has $r_1$ number of pairs of $\{e\}$-horizontal edges and $r_2$ number of $\mathbb{Z}_2$-horizontal edges, $(\widetilde{d}, \widetilde{P})$ is the top row of $\widetilde{d}, \widetilde{P}$ is a subset of $\widetilde{d}$ such that $|\widetilde{P}| = 2s_1+s_2 \Big\}.$

 Also, $\left| \widetilde{\mathbb{J}}_{2s_1 + s_2}^{2 r_1 + r_2} \right|  =  \widetilde{f}^{2r_1+r_2}_{2s_1+s_2}$ and $\left| \widetilde{J}^{2k}_{2s_1+s_2}\right| = \widetilde{f}_{2s_1+s_2}.$

\item[(c)] For $ 0 \leq r \leq k-s,  0 \leq s \leq k$  put $ J^{k}_{s} = \underset{0 \leq r \leq k-s }{\cup} \ \  \mathbb{J}_{s}^{r}$ where

$\mathbb{J}_{s}^{r} = \Big\{ R^d \in I^{k}_{s} \ \Big| \ R^d = U^{R^d}_{R^d}, \left( R^d\right)^+ \text{ and }  \left(R^d\right)^{-} \text{ are the same}, \sharp^p(U^{R^d}_{R^d}) = s,  U^{R^{d}}_{R^d} \text{ has }$

 $ \hspace{9cm} r \text{ number of horizontal edges} \Big\}$

 Also, $ \left| \mathbb{J}_{s}^{r} \right|  =   = f^{r}_{s} \text{ and }
  \left|  J^{k}_{s} \right|  =  f_s.
$

          \end{enumerate}

\end{notation}

\begin{defn} \label{D2.13} \textbf{(\cite{KP}, Definition 3.6)}
\begin{enumerate}
  \item[(i)] The diagrams in $J^{2k}_{2s_1 + s_2}\left(\widetilde{J}^{2k}_{2s_1 + s_2}\right)$ are indexed as follows:

\centerline{$\left\{U^{(d_i, P_i)}_{(d_i, P_i)} \ \Big| \ 1 \leq i \leq f_{2s_1 + s_2}\right\}_{U^{(d_i, P_i)}_{(d_i, P_i)} \in J^{2k}_{2s_1+s_2}} \left(\left\{U^{(\widetilde{d}_i, \widetilde{P}_i)}_{(\widetilde{d}_i, \widetilde{P}_i)} \ \Big| \ 1 \leq i \leq \widetilde{f}_{2s_1 + s_2}\right\}_{U^{(\widetilde{d}_i, \widetilde{P}_i)}_{(\widetilde{d}_i, \widetilde{P}_i)} \in \widetilde{J}^{2k}_{2s_1+s_2}}\right).$}

$i < j,$

\begin{itemize}
  \item[(1)] if $2r_1+r_2 < 2r'_1+r'_2$
  \item[(2)] if $2r_1+r_2 = 2r'_1+r'_2 $ and $r_1+r_2 < r'_1+r'_2$
  \item[(3)] if  $2r_1+r_2 = 2r'_1+r'_2$ and $r_1+r_2 = r'_1+r'_2$ then it can be indexed arbitrarily.
\end{itemize}
where $r_1$ is the number of pairs of $\{e\}$-horizontal edges in $U^{(\widetilde{d}_i, \widetilde{P}_i)}_{(\widetilde{d}_i, \widetilde{P}_i)}\left( U^{(d_i, P_i)}_{(d_i, P_i)}\right), r'_1$ is the number of pairs of $\{e\}$-horizontal edges in $U^{(\widetilde{d}_j, \widetilde{P}_j)}_{(\widetilde{d}_j, \widetilde{P}_j)}\left( U^{(d_j, P_j)}_{(d_j, P_j)}\right), r_2$ is the number of $\mathbb{Z}_2$-horizontal edges in $U^{(\widetilde{d}_i, \widetilde{P}_i)}_{(\widetilde{d}_i, \widetilde{P}_i)}\left( U^{(d_i, P_i)}_{(d_i, P_i)}\right)$ and $r'_2$ is the number of $\mathbb{Z}_2$-horizontal edges in $U^{(\widetilde{d}_j, \widetilde{P}_j)}_{(\widetilde{d}_j, \widetilde{P}_j)}\left( U^{(d_j, P_j)}_{(d_j, P_j)}\right).$
  \item[(ii)]  The diagrams in $J^{k}_{s}$ are indexed as follows:

\centerline{$\left\{U^{R^{d_i}}_{R^{d_i}} \ \Big| \ 1 \leq i \leq f_{s}\right\}_{U^{R^{d_i}}_{R^{d_i}} \in J_s^k}$}

$i < j,$
\begin{enumerate}
  \item[(1)] if $r < r'$,
  \item[(2)] if $r = r'$ then it can be indexed arbitrarily
\end{enumerate}

where $r(r')$ is the number of horizontal edges in $U^{R^{d_i}}_{R^{d_i}}\left( U^{R^{d_j}}_{R^{d_j}}\right).$

\end{enumerate}

\end{defn}

\begin{defn} \label{D2.14} \textbf{(\cite{KP}, Definition 3.7)}
\begin{enumerate}
  \item[(a)]\textbf{Gram matrices of the algebra of $\mathbb{Z}_2$-relations:}

  \NI For $0 \leq s_1, s_2 \leq k, 0 \leq s_1+s_2 \leq k,$ define Gram matrices of the algebra of $\mathbb{Z}_2$-relations $G_{2s_1+s_2}$ as follows:

\centerline{$G_{2s_1 + s_2} = \left( A_{2r_1 + r_2, 2r'_1+r'_2}\right)_{\substack{\hspace{-1.5cm}0 \leq r_1+r_2, r'_1 + r'_2 \leq k-s_1-s_2 \\ 0 \leq r_1, r'_1 \leq k-s_1-s_2, 0 \leq r_2, r'_2 \leq k-s_1-s_2}} $}
\NI where $A_{2r_1 + r_2, 2r'_1 + r'_2}$ denotes the block matrix whose entries are $a_{ij}$ with
\begin{center}
$\begin{array}{lllll}
  a_{ij} & = & x^{l(P_i \vee P_j)}  &  if &  \sharp^p \left( U^{(d_i, P_i)}_{(d_i, P_i)} U^{(d_j, P_j)}_{(d_j, P_j)} \right) = 2s_1 + s_2 \\
      a_{ij} & = & 0 & \text{Otherwise } i.e., &  \sharp^p \left( U^{(d_i, P_i)}_{(d_i, P_i)} U^{(d_j, P_j)}_{(d_j, P_j)}\right) < 2s_1 + s_2, \\
 \end{array}
$
\end{center}
 $ l(P_i \vee P_j) = l \left( U^{(d_i, P_i)}_{(d_i, P_i)} U^{(d_j, P_j)}_{(d_j, P_j)}\right)  $ where $l(P_i \vee P_j)$ denotes the number of connected components in $d_i. d_j$ excluding the union of all the connected components of $P_i$ and $P_j$ or equivalently, $ l \left( U^{(d_i, P_i)}_{(d_i, P_i)} U^{(d_j, P_j)}_{(d_j, P_j)}\right)$ is the number of loops which lie in the middle row when $U^{(d_i, P_i)}_{(d_i, P_i)}$ is multiplied with $U^{(d_j, P_j)}_{(d_j, P_j)}$, $U^{(d_i, P_i)}_{(d_i, P_i)} \in \mathbb{J}^{2r_1 + r_2}_{2s_1 +s_2}$ and $U^{(d_j, P_j)}_{(d_j, P_j)} \in \mathbb{J}^{2r'_1+ r'_2}_{2s_1 +s_2}$ respectively.

  \item[(b)]\textbf{Gram matrices of signed partition algebra:}

   \NI For $0 \leq s_1 \leq k, 0 \leq s_2 \leq k-1, 0 \leq s_1+s_2 \leq k-1,$ define Gram matrix of the signed partition algebra $\widetilde{G}_{2s_1+s_2}$  as follows:

\centerline{$\widetilde{G}_{2s_1 + s_2} = \left( \widetilde{A}_{2r_1 + r_2, 2r'_1+r'_2}\right)_{\substack{\hspace{-1.5cm}0 \leq r_1+r_2, r'_1 + r'_2 \leq k-s_1-s_2    \\ 0 \leq r_1, r'_1 \leq k-s_1-s_2, 0 \leq r_2, r'_2 \leq k-s_1-s_2-1}}$} where $\widetilde{A}_{2r_1 + r_2, 2r'_1 + r'_2}$ denotes the block matrix whose entries are $a_{ij}$ with
\begin{center}
$\begin{array}{lllll}
  a_{ij} & = & x^{l(\widetilde{P}_i \vee \widetilde{P}_j)}  &  if &  \sharp^p \left( U^{(\widetilde{d}_i, \widetilde{P}_i)}_{(\widetilde{d}_i, \widetilde{P}_i)} U^{(\widetilde{d}_j, \widetilde{P}_j)}_{(\widetilde{d}_j, \widetilde{P}_j)} \right) = 2s_1 + s_2 \\
      a_{ij} & = & 0 & \text{Otherwise } i.e., &  \sharp^p \left( U^{(\widetilde{d}_i, \widetilde{P}_i)}_{(\widetilde{d}_i, \widetilde{P}_i)} U^{(\widetilde{d}_j, \widetilde{P}_j)}_{(\widetilde{d}_j, \widetilde{P}_j)}\right) < 2s_1 + s_2, \\
 \end{array}
$
\end{center}
 $ l(\widetilde{P}_i \vee \widetilde{P}_j) = l \left( U^{(\widetilde{d}_i, \widetilde{P}_i)}_{(\widetilde{d}_i, \widetilde{P}_i)} U^{(\widetilde{d}_j, \widetilde{P}_j)}_{(\widetilde{d}_j, \widetilde{P}_j)}\right)  $ where $l(\widetilde{P}_i \vee \widetilde{P}_j)$ denotes the number of connected components in $\widetilde{d}_i. \widetilde{d}_j$ excluding the union of all the connected components of $\widetilde{P}_i$ and $\widetilde{P}_j$ or equivalently, $ l \left( U^{(\widetilde{d}_i, \widetilde{P}_i)}_{(\widetilde{d}_i, \widetilde{P}_i)} U^{(\widetilde{d}_j, \widetilde{P}_j)}_{(\widetilde{d}_j, \widetilde{P}_j)}\right)$ is the number of loops which lie in the middle row when $U^{(\widetilde{d}_i, \widetilde{P}_i)}_{(\widetilde{d}_i, \widetilde{P}_i)}$ is multiplied with $U^{(\widetilde{d}_j, \widetilde{P}_j)}_{(\widetilde{d}_j, \widetilde{P}_j)}$, $U^{(\widetilde{d}_i, \widetilde{P}_i)}_{(\widetilde{d}_i, \widetilde{P}_i)} \in \mathbb{J}^{2r_1 + r_2}_{2s_1 +s_2}$ and $U^{(\widetilde{d}_j, \widetilde{P}_j)}_{(\widetilde{d}_j, \widetilde{P}_j)} \in \mathbb{J}^{2r'_1+ r'_2}_{2s_1 +s_2}$ respectively.

  \item[(c)]\textbf{Gram matrices of partition algebra:}

  \NI For $0 \leq s \leq k,$ define Gram matrix of the partition algebra $G_s$   as follows:

\centerline{$G_{s} = \left( A_{r, r'}\right)_{0 \leq r, r' \leq k-s} $}
\NI where $A_{r, r'}$ denotes the block matrix whose entries are $a_{ij}$ with
\begin{center}
$\begin{array}{lllll}
  a_{ij} & = & x^{l\left(R^{d_i} R^{d_j}\right)}  &  if &  \sharp^p \left( U^{R^{d_i}}_{R^{d_i}} U^{R^{d_j}}_{R^{d_j}} \right) = s\\
      a_{ij} & = & 0 & Otherwise & i.e., \sharp^p \left( U^{R^{d_i}}_{R^{d_i}} U^{R^{d_j}}_{R^{d_j}}\right) < s, \\
 \end{array}
$
\end{center}
 $ l\left(R^{d_i} R^{d_j}\right) = l \left( U^{R^{d_i}}_{R^{d_i}} U^{R^{d_j}}_{R^{d_j}}\right)  $ where $l(d_i  d_j)$ denotes the number of connected components which lie in the middle row while multiplying  $U^{R^{d_i}}_{R^{d_i}}$ with $ U^{R^{d_j}}_{R^{d_j}},$ $U^{R^{d_i}}_{R^{d_i}} \in \mathbb{J}^{r}_{s}$ and $U^{R^{d_j}}_{R^{d_j}} \in \mathbb{J}^{r'}_{s}$ respectively.

\end{enumerate}

\end{defn}
\begin{notation} \label{N2.15} \textbf{(\cite{KP}, Notation 3.21)}
\begin{enumerate}
  \item[(a)] Let $U^{(d_i, P_i)}_{(d_i, P_i)}, U^{(d_j, P_j)}_{(d_j, P_j)} \in \mathbb{J}^{2r_1+r_2}_{2s_1+s_2}$ such that $\sharp^p \left( U^{(d_i, P_i)}_{(d_i, P_i)}. U^{(d_j, P_j)}_{(d_j, P_j)}\right) < 2s_1 + s_2,$ so that the $ij$-entry of the block matrix $A_{2r_1 + r_2, 2r_1 + r_2}$ is zero and $0 \leq r_1 \leq k - s_1 - s_2 , 0 \leq r_2 < k-s_1-s_2, 2r_1+r_2 \leq 2k - 2s_1-s_2.$

 Put $U^{(d_i, P_i)}_{(d_i, P_i)} = U^{l^1_{f}}_{l^1_{f}} \otimes U^{d_{i}-f}_{d_i-f}$ and $U^{(d_j, P_j)}_{(d_j, P_j)} = U^{l^2_{f}}_{l^2_{f}} \otimes U^{d_{j}-f}_{d_j-f}$ where $U^{l^1_{f}}_{l^1_{f}} \left(U^{l^2_{f}}_{l^2_{f}} \right)$ is the sub diagram of $U^{(d_i, P_i)}_{(d_i, P_i)} \left(U^{(d_j, P_j)}_{(d_j, P_j)} \right), U^{l^1_f}_{l^1_f}, U^{l^2_f}_{l^2_f} \in \mathbb{J}^{2t_1+t_2}_{2t_1+t_2}$ and every $\{e\} \left( \mathbb{Z}_2 \right)$ through class of $U^{l^1_f}_{l^1_f}$ is replaced by a $\{e\} \left( \mathbb{Z}_2 \right)$ horizontal edge and vice versa.
  \item[(b)] Let $U^{(\widetilde{d}_i, \widetilde{P}_i)}_{(\widetilde{d}_i, \widetilde{P}_i)}, U^{(\widetilde{d}_j, \widetilde{P}_j)}_{(\widetilde{d}_j, \widetilde{P}_j)} \in \widetilde{\mathbb{J}}^{2r_1+r_2}_{2s_1+s_2}$ such that $\sharp^p \left( U^{(\widetilde{d}_i, \widetilde{P}_i)}_{(\widetilde{d}_i, \widetilde{P}_i)}. U^{(\widetilde{d}_j, \widetilde{P}_j)}_{(\widetilde{d}_j, \widetilde{P}_j)}\right) < 2s_1 + s_2,$ so that the $ij$-entry of the block matrix $\widetilde{A}_{2r_1 + r_2, 2r_1 + r_2}$ is zero and $0 \leq r_1 \leq k - s_1 - s_2 -1, 0 \leq r_2 < k-s_1-s_2-1, 2r_1+r_2 \leq 2k - 2s_1-s_2-1.$

 Put $U^{(\widetilde{d}_i, \widetilde{P}_i)}_{(\widetilde{d}_i, \widetilde{P}_i)} = \widetilde{U}^{l^1_{f}}_{l^1_{f}} \otimes \widetilde{U}^{d_{i}-f}_{d_i-f}$ and $U^{(\widetilde{d}_j, \widetilde{P}_j)}_{(\widetilde{d}_j, \widetilde{P}_j)} = \widetilde{U}^{l^2_{f}}_{l^2_{f}} \otimes \widetilde{U}^{d_{j}-f}_{d_j-f}$ where $\widetilde{U}^{l^1_{f}}_{l^1_{f}} \left(\widetilde{U}^{l^2_{f}}_{l^2_{f}} \right)$ is the sub diagram of $U^{(\widetilde{d}_i, \widetilde{P}_i)}_{(\widetilde{d}_i, \widetilde{P}_i)} \left(U^{(\widetilde{d}_j, \widetilde{P}_j)}_{(\widetilde{d}_j, \widetilde{P}_j)} \right), \widetilde{U}^{l^1_f}_{l^1_f}, \widetilde{U}^{l^2_f}_{l^2_f} \in \widetilde{\mathbb{J}}^{2t_1+t_2}_{2t_1+t_2}$ and every $\{e\} \left( \mathbb{Z}_2 \right)$ through class of $\widetilde{U}^{l^1_f}_{l^1_f}$ is replaced by a $\{e\} \left( \mathbb{Z}_2 \right)$ horizontal edge and vice versa.
 \item[(c)]Let $U^{R^{d_i}}_{R^{d_i}}, U^{R^{d_j}}_{R^{d_j}} \in \mathbb{J}^{r}_{s}$ such that $\sharp^p \left( U^{R^{d_i}}_{R^{d_i}}. U^{R^{d_j}}_{R^{d_j}}\right) < s,$ so that the $ij$-entry of the block matrix $A_{r, r}$ is zero and $0 \leq r \leq k - s.$

 Put $U^{R^{d_i}}_{R^{d_i}} = U^{l_1}_{l_1} \otimes U^{d_{i}-l_1}_{d_i-l_1}$ and $U^{R^{d_j}}_{R^{d_j}} = U^{l_2}_{l_2} \otimes U^{d_{j}-l_2}_{d_j-l_2}$ where $U^{l_1}_{l_1} \left(U^{l_2}_{l_2} \right)$ is the sub diagram of $U^{R^{d_i}}_{R^{d_i}} \left(U^{R^{d_j}}_{R^{d_j}} \right), U^{l_1}_{l_1}, U^{l_2}_{l_2} \in \mathbb{J}^{t}_{t}$ and every through class of $U^{l_1}_{l_1}$ is replaced by a  horizontal edge and vice versa.
                    \end{enumerate}
\end{notation}

\begin{prop} \label{P2.16}\textbf{(\cite{KP}, Proposition 3.27)}

Let $U^{(d_i, P_i)}_{(d_i, P_i)}, U^{(d_j, P_j)}_{(d_j, P_j)}, U^{(\widetilde{d}_i, \widetilde{P}_i)}_{(\widetilde{d}_i, \widetilde{P}_i)}, U^{(\widetilde{d}_j, \widetilde{P}_j)}_{(\widetilde{d}_j, \widetilde{P}_j)}$ and $U^{R^{d_i}}_{R^{d_i}}, U^{R^{d_j}}_{R^{d_j}}$ be as in Notation \ref{N2.15}.
\begin{enumerate}
\item[(a)] After performing the column operations  to eliminate the non-zero entries which lie above corresponding to the diagrams coarser than $U^{(d_j, P_j)}_{(d_j, P_j)}$,

    \item[(b)] After performing the column operations  to eliminate the non-zero entries which lie above corresponding to the diagrams coarser than $U^{(\widetilde{d}_j, \widetilde{P}_j)}_{(\widetilde{d}_j, \widetilde{P}_j)}$,

     then the $ij$-entry of the block matrix $A_{2r_1+r_2, 2r_1+r_2}$ for $0 \leq r_1+r_2, r_1, r_2 \leq k - s_1-s_2$ and the block matrix $\widetilde{A}_{2r_1+r_2, 2r_1+r_2}$ for $0 \leq r_1+r_2, r_1, r_2 \leq k - s_1-s_2-1$ is replaced by
\begin{enumerate}
  \item[(i)] $(-1)^{t_1+t_2}\ (t_1)! \ (t_2)! \ 2^{t_1} \underset{j=t_1}{\overset{r_1-1}{\prod}}[x^2-x-2(s_1+j)] \underset{m=t_2}{\overset{r_2-1}{\prod}} [x-(s_2+m)]$ \ \ \ \ if $r_1 \geq 1$ and $r_2 \geq 1,$
  \item[(ii)] $(-1)^{t_2} \  \ (t_2)! \ \underset{m=t_2}{\overset{r_2-1}{\prod}} [x-(s_2+m)]$ \ \ \ \ if $r_1 = 0$ and $r_2 \neq 0,$
  \item[(iii)] $(-1)^{t_1}\ (t_1)!  \ 2^{t_1} \underset{j=t_1}{\overset{r_1-1}{\prod}}[x^2-x-2(s_1+j)]$ \ \ \ \ if $r_1 \neq 0$ and $r_2 = 0,$
\end{enumerate}
\item[(c)] After performing the column operations  to eliminate the non-zero entries which lie above corresponding to the diagrams coarser than $U^{R^{d_j}}_{R^{d_j}}$, then the $ij$-entry is replaced by

    \centerline{$(-1)^{t}\ t!  \underset{j=t}{\overset{r-1}{\prod}} [x-(s+l)].$}
\end{enumerate}
\end{prop}

\begin{thm}\label{T2.17}\textbf{(\cite{KP}, Theorem 3.29)}
\begin{enumerate}
\item[(a)] Let $G'_{2s_1+s_2}$ be the matrix similar to the Gram matrix $G_{2s_1+s_2}$ of the algebra of $\mathbb{Z}_2$-relations which is obtained after the column operations and the corresponding row operations on $G_{2s_1+s_2}.$ Then

 \centerline{$G'_{2s_1 + s_2} = \left(\underset{\substack{0 \leq r_1 \leq k-s_1-s_2 \\ 0 \leq r_2 < k-s_1-s_2 \\ 2r_1+r_2 \leq 2k - 2s_1 - 2s_2}}{\bigoplus} A'_{2r_1+r_2, 2r_1+r_2} \right)$}

 \item[(b)]Let $\widetilde{G}'_{2s_1+s_2}$ be the matrix similar to the Gram matrix $\widetilde{G}_{2s_1+s_2}$ of signed partition algebras which is obtained after the column operations and the corresponding row operations on $\widetilde{G}_{2s_1+s_2}.$ Then

 \centerline{$\widetilde{G}'_{2s_1 + s_2} = \left(\underset{\substack{0 \leq r_1 \leq k-s_1-s_2-1 \\ 0 \leq r_2 < k-s_1-s_2-1 \\ 2r_1+r_2 \leq 2k - 2s_1 - 2s_2-1}}{\bigoplus} \widetilde{A}'_{2r_1+r_2, 2r_1+r_2} \right)\bigoplus \widetilde{A}'_{\lambda'}$}
\NI where
   \begin{enumerate}
   \item[(i)] the diagonal element of $A'_{2r_1+r_2, 2r_1+r_2}\left( \widetilde{A}'_{2r_1+r_2, 2r_1+r_2}\right)$ is given by

$\begin{array}{llll}
   1. & \underset{j=0}{\overset{r_1-1}{\prod}} [x^2-x-2(s_1+j)] \underset{m=0}{\overset{r_2-1}{\prod}} [x-(s_2+m)] &  if & r_1 \geq 1, r_2 \geq 1 \\
   2. & \underset{j=0}{\overset{r_1-1}{\prod}} [x^2-x-2(s_1+j)]  & if & r_2 = 0 \\
   3. & \underset{m=0}{\overset{r_2-1}{\prod}} [x-(s_2+m)] & if &  r_1 = 0
 \end{array}
$
     \item[(ii)] The entry $b_{ij}$ of the block matrix $A'_{2r_1 +r_2,2r_1+r_2}\left( \widetilde{A}'_{2r_1+r_2, 2r_1+r_2}\right)$ is replaced by

\hspace{-1cm}$\begin{array}{llll}
  1. & (-1)^{t_1 + t_2} \ 2^{t_1} \ (t_1)! \ (t_2)! \underset{j=0}{\overset{r_1-t_1-1}{\prod}} [x^2-x-2(s_1 + t_1+ j)] \underset{m=0}{\overset{r_2-t_2-1}{\prod}} [x-(s_2 + t_2 + m)]  & & \\
  &\hspace{9cm}\text{ if } r_1 \geq 1, r_2 \geq 1 & & \\
  2. & (-1)^{t_1 } \ 2^{t_1} \ (t_1)! \ \underset{j=0}{\overset{r_1-t_1-1}{\prod}} [x^2-x-2(s_1 + t_1+ j)] \ \ \ \ \ \ \ \ \text{ if } \ \ r_2 = 0 \\
  3. & (-1)^{t_2} \  (t_2)! \underset{m=0}{\overset{r_2-t_2-1}{\prod}} [x-(s_2 + t_2 + m)]  \ \ \ \ \hspace{2cm} \ \text{ if } \ \  r_1 = 0
 \end{array}
$

 \NI whenever $U^{(d_i, P_i)}_{(d_i, P_i)}$ and $U^{(d_j, P_j)}_{(d_j, P_j)}$ can be defined  as in Notation \ref{N2.15}(a), (b) and Proposition \ref{P2.16}(a), (b).

\item[(iii)] All other entries of the block matrix $A'_{2r_1+r_2, 2r_1+r_2}\left( \widetilde{A}'_{2r_1+r_2, 2r_1+r_2}\right)$ are zero.
   \end{enumerate}

  The underlying partitions of the diagrams corresponding to the entries of the block matrix $\widetilde{A}'_{2r_1+r_2, 2r_1+r_2}$ are $\lambda = [\lambda_1^2]^1 [2 \lambda_2]^2 [\lambda_3^2] [2 \lambda_4]^4$ with $\lambda_1^2 = \left( \lambda_{11}^2, \cdots, \lambda_{1s_1}^2\right), 2\lambda_2 = \left( 2 \lambda_{21}, \cdots, \lambda_{2s_2} \right), \lambda_{3}^2 = \left(\lambda^2_{31}, \cdots, \lambda^2_{3r_1} \right), 2 \lambda_{4} = \left( 2 \lambda_{41}, \cdots, 2 \lambda_{4r_2}\right)$ such that atleast one of $\lambda_{1i}, \lambda_{2j}, \lambda_{3f}, \lambda_{4m}$ is greater than $1$  for $1 \leq i \leq s_1, 1 \leq j \leq s_2, 1 \leq f \leq r_1$ and $1 \leq m\leq r_2.$


 \item[(b)']Let $\widetilde{A}'_{\lambda'}$ where  the partition $\lambda' = ([1^{s_1}]^1, [1^{s_2}]^2, [1^{r_1}]^3, [1^{r_2}]^4)$  such that $s_1+s_2+r_1+r_2 = k, r_1 > 0$
and $\widetilde{A}'_{\lambda'}$ is the block sub matrix corresponding to the diagrams whose underlying partition is $\lambda'.$

\begin{enumerate}
  \item[(i)] The $ii$-entry $x^{2r'_1+r'_2}$ of the matrix $\widetilde{A}'_{\lambda'}$  is replaced by

\centerline{$\underset{j=0}{\overset{r'_1-1}{\prod}} [x^2-x-2(s_1+j)] \underset{m=0}{\overset{r'_2-1}{\prod}} [x-(s_2+m)] + \underset{m=0}{\overset{k-s_1-s_2-1}{\prod}} [x-(s_2+m)]$}

\NI where $1 \leq r'_1 \leq k - s_1 - s_2 $ and  $r'_2 = k - s_1 - s_2 - r'_1.$
  \item[(ii)] The zero in the $ij$-entry  is replaced by

  \centerline{$(-1)^{t_1+t_2} \ 2^{t_1} \ (t_1)! \ (t_2)! \ \underset{j=0}{\overset{r'_1-t_1-1}{\prod}} [x^2-x-2(s_1+t_1+j)]  \underset{m=0}{\overset{r'_2-t_2-1}{\prod}} [x-(s_2+m+t_2)] + \underset{m=0}{\overset{k-s_1-s_2-1}{\prod}} [x-(s_2+m)]$}
  \NI where $U^{(\widetilde{d}_i, \widetilde{P}_i)}_{(\widetilde{d}_i, \widetilde{P}_i)}$ and $U^{(\widetilde{d}_j, \widetilde{P}_j)}_{(\widetilde{d}_j, \widetilde{P}_j)}$  are as in Notation \ref{N2.15}(b) and Proposition \ref{P2.16}(b) where $1 \leq i, j \leq 2k - 2s_1 - 2s_2$ and $i \neq j.$

  \item[(iii)] If $\sharp^p \left(U^{(\widetilde{d}_i, \widetilde{P}_i)}_{(\widetilde{d}_i, \widetilde{P}_i)} U^{(\widetilde{d}_j, \widetilde{P}_j)}_{(\widetilde{d}_j, \widetilde{P}_j)}\right) = 2s_1 + s_2$ then the  $ij$-entry is  replaced by

  \centerline{$ (-1)^{r_1 + r'_1} \ \underset{m=0}{\overset{k-s_1-s_2-1}{\prod}} [x-(s_2+m)]$}
where $U^{(\widetilde{d}_i, \widetilde{P}_i)}_{(\widetilde{d}_i, \widetilde{P}_i)} \in \mathbb{J}_{2s_1+s_2}^{2r'_1+k-s_1-s_2-r'_1}$ and $U^{(\widetilde{d}_j, \widetilde{P}_j)}_{(\widetilde{d}_j, \widetilde{P}_j)} \in \mathbb{J}_{2s_1+s_2}^{2r_1 + k - s_1-s_2 -r_1}$ where $1 \leq i, j \leq 2k - 2s_1 - 2s_2$ and $i \neq j.$

\NI All other entries of the block matrix $\widetilde{A}'_{\lambda'}$ are zero.
\end{enumerate}
\item[(c)] Let $G'_{s}$ be the matrix similar to the Gram matrix $G_{s}$ which is obtained after the column operations and the row operations on $G_{s}.$ Then

 \centerline{$G'_{s} = \left(\underset{0 \leq r \leq k-s}{\bigoplus} A'_{r, r} \right)$}
\NI where
   \begin{enumerate}
   \item[(i)] the diagonal element of $A'_{r, r}$ is given by

 \centerline{$ \underset{l=0}{\overset{r-1}{\prod}} [x-(s+l)]$}
     \item[(ii)] The entry $b_{ij}$ of the block matrix $A'_{r, r}$ is replaced by

     \centerline{$ (-1)^{t}  \ (t)! \  \underset{j=t}{\overset{r-1}{\prod}}  [x-(s  + l)] $} whenever $U^{R^{d_i}}_{R^{d_i}}$ and $U^{R^{d_j}}_{R^{d_j}}$ can be defined  as in  of Notation \ref{N2.15}(c) and Proposition \ref{P2.16}(c).

\item[(iii)] All other entries of the block matrix $A'_{r, r}$ are zero.
   \end{enumerate}

\end{enumerate}
\end{thm}

\section{\textbf{Diagram Matrix}}

In this section, we introduce the symmetric diagram matrices $A^{s+r, s}$ of size ${_{(s+r)}}C_s$ based on the diagrams. By induction, we compute the eigenvalues of the symmetric diagram matrices $A^{s+r, s}$ using elementary row and column operations.





\begin{defn}\label{D3.2}
Fix $s$ and $r.$ Choose a $d^+ \in R_k$ such that $d^+$ has $s+r$-connected components where $R_k$ denote the set of all equivalence relations on $\underline{k} = \{1, 2, \cdots, k\}.$

We shall draw a diagram $d$ graphically using the graph $d^+$  with $s$ through classes as follows:
\begin{enumerate}
  \item[(i)] Draw $d^+$ in the top row and a copy of $d^+$ denoted by $d^-$ in the bottom row.
  \item[(ii)] Among the $s+r$ connected components in the top row, choose $s$ connected components and join each connected component with the respective connected component in the bottom row by vertical edges.
  \item[(iii)] The resultant diagram is denoted by $d.$
\end{enumerate}
\end{defn}

\begin{defn}\label{D3.3}
Let $\Omega^{s+r, s}$ denote the collection of all diagrams with $s$ through classes mentioned as above in Definition \ref{D3.2} and the number of such diagrams are denoted by  $\left| \Omega^{s+r, s}\right| = {_{(s+r)}}C_s.$
\end{defn}

\begin{ex}\label{E3.4}
Let $s+r = 9$ and $\rho = \left( \{1,2,3\}, \{4, 5\}, \{6,7\}, \{8\}, \{9\}\right).$ The diagrams with three through classes corresponding to the set partition $\rho$ are
\begin{center}
\includegraphics{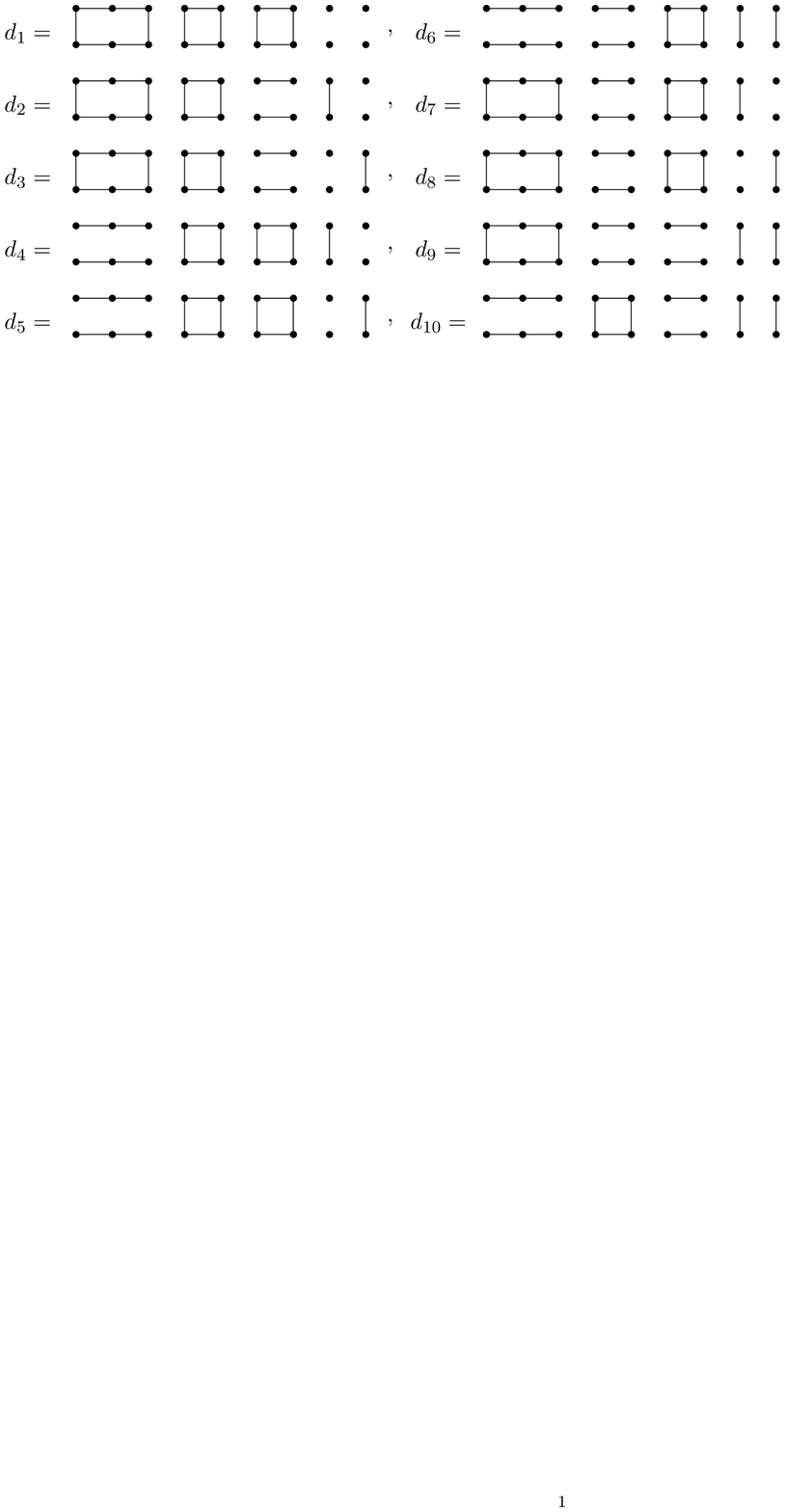}
\end{center}
\end{ex}

\begin{defn}\label{D3.5}
Fix $s$ and $r.$ Consider the following variables $\{x_0, x_1, \cdots, x_{\text{min}\{s, r\}}\}.$ Define a matrix $A^{s+r, s}$ of size ${_{(s+r)}}C_s$ with the entries $\{x_0, x_1, \cdots, x_{\text{min}\{s,r\}}\}$ as follows:

\centerline{$A^{s+r, s} = \left( a_{ij}\right)_{{_{(s+r)}}C_s \times {_{(s+r)}}C_s}$ with $a_{ij} = x_{\text{min}\{s,r\}-f}$}
\NI where $f$ denotes the number of horizontal edges in $d_i$ which are replaced by through classes in $d_j$ and vice versa, $d_i, d_j \in \Omega^{s+r, s}.$
\end{defn}

\begin{rem}\label{R3.6}
For the sake of convenience, we shall replace every through class in the above mentioned diagram by a through class $\left( \ \big | \ \right)$  obtained by joining a single vertex in the top row and the corresponding vertex in the bottom row. Similarly, every horizontal edge in the top and bottom row is replaced by a vertex $( \centerdot )$ respectively.
\end{rem}

\begin{lem}\label{L3.7}
The matrix $A^{s+r, s}$ is symmetric.
\end{lem}

\begin{proof}
Since the top and bottom row of the diagrams in $\Omega^{s+r, s}$ are the same, the matrix $A^{s+r, s}$ is symmetric.
\end{proof}

\begin{defn}\label{D3.8}
The matrix $A^{s+r, s}$ is called as \textbf{symmetric diagram matrix}.
\end{defn}

\begin{rem}
The number of times $x_{\text{min}\{s, r\}-t}$ occurs in any row or column of a symmetric diagram matrix $A^{s+r, s}$ of size ${_{(s+r)}}C_s$ is given by ${_s}C_t \ {_r}C_t.$
\end{rem}

The eigenvalues of the symmetric diagram matrix $A^{s+r, s}$ can be obtained using induction. The main theorem of this paper gives all the distinct eigenvalues of the symmetric diagram matrix $A^{s+r, s}.$
\subsection{Main Theorem}
\begin{thm}\label{T3.9}
The set of all distinct eigenvalues of the symmetric diagram matrix $A^{s+r, s}$ of size ${_{(s+r)}}C_s$ with entries $\{x_0, x_1, \cdots, x_{\text{min}\{s, r\}}\}$ are given by

\begin{equation}\label{E3.1}
\underset{t=0}{\overset{\text{min}\{s,r\}}{\sum}}\left[ \underset{j=0}{\overset{l}{\sum}} (-1)^j \ {_l}C_j \ {_{(s-l)}}C_{(t-j)} \ {_{(r-l)}}C_{(t-j)}\right] x_{\text{min}\{s,r\}-t}
\end{equation}
\NI for all $0 \leq l \leq \text{min }\{s,r\}.$
\end{thm}

\begin{lem}\label{L3.10}
If $a_0, a_1, \cdots, a_r$ are $(r+1)$-variables and if $a_t^l = \underset{j=0}{\overset{l}{\sum}} (-1)^j \ {_{l}}C_j \ a_{t-j}$ then

\centerline{$a_t^{l+1} = a_t^l - a_{t-1}^l.$}
\end{lem}
\begin{proof}
Consider
\begin{eqnarray*}
  a_t^l - a_{t-1}^l &=& \underset{j=0}{\overset{l}{\sum}} (-1)^j \ {_{l}}C_j \ a_{t-j} - \underset{j=0}{\overset{l}{\sum}} (-1)^j \ {_{l}}C_j \ a_{t-1-j} \\
   &=& \underset{j=0}{\overset{l}{\sum}} (-1)^j \ {_{l}}C_j \ a_{t-j} - \underset{j=1}{\overset{l+1}{\sum}} (-1)^{j'-1} \ {_{l}}C_{j'-1} \ a_{t-j'} \\
   &=& a_t + \underset{j=1}{\overset{l}{\sum}} (-1)^j \ \left[ {_{l}}C_j + {_l}C_{j-1} \right] \ a_{t-j} + (-1)^{l + 1} \  a_{t-(l+1)} \\
   &=& \underset{j=0}{\overset{l+1}{\sum}} (-1)^j \  {_{(l+1)}}C_j  \ a_{t-j} \\
   &=&  a_t^{l+1}
\end{eqnarray*}
\end{proof}

\begin{cor}\label{C3.11}
In particular, if $a_i = {_{(s-l)}}C_i \ {_{(r-l)}}C_i$ for $0 \leq i \leq \text{ min }\{s-l, r-l\}$ then  $a_t^{l+1} = \underset{j=0}{\overset{l+1}{\sum}}(-1)^j \ {_{l+1}}C_j \ {_{(s-(l+1))}}C_{(t-j)} \ {_{(r-(l+1))}}C_{(t-j)}.$
\end{cor}
\begin{proof}
The proof follows from induction and Lemma \ref{L3.10}.
\end{proof}
\subsection{Arrangement of Diagrams in $\Omega^{s+r, s}$}

\begin{defn}\label{D3.12}

Let $I, J \subset \{1, 2, \cdots, s+r\}$ with $I \cap J = \phi.$ Define,

$\Omega_{I, J} = \{ d \in \Omega^{s+r, s} \ | \ \text{ the } i^{\text{th}} \text{ connected component of } d \text{ is a through class } \left( \  | \ \right) $ for all $i \in I$ and the $j^{\text{th}}$ connected component of $d$ consists of two dots for all $ j \in J\}.$

In particular, if $I = \{1\}$ and $J = \{s+r\}$ then

$\Omega_{\{1\}, \{s+r\}} = \{ d \in \Omega^{s+r, s} \ | \ $ the first connected component of $d$ is a through class $\left( \ | \ \right)$ and the $(s+r)^{\text{th}}$ connected component consists of two dots $\}.$

\end{defn}

\begin{lem}\label{L3.13}
Let $\Omega_{I, J}$ be as in Definition \ref{D3.12}. Then

\centerline{$\left| \Omega_{I, J}\right| = {_{\left( s+r - |I| - |J|\right)}}C_{s - |I|}.$}

In particular,

\centerline{$\left| \Omega_{\{1\}, \{s+r\}}\right| = {_{(s+r-2)}}C_{(s-1)}.$}
\end{lem}
\begin{proof}
The proof follows from the Definition of $\Omega_{I, J}.$
\end{proof}

We shall arrange the diagrams in $\Omega^{s+r, s}$ as follows:

\begin{enumerate}
  \item[(i)]  The collection of diagrams whose  first connected component is a through class and the last connected component is a horizontal edge is denoted by $\Omega_{\{1\}, \{s+r\}}$. Also, $\left| \Omega_{\{1\}, \{s+r\}} \right| = {_{(s+r-2)}}C_{(s-1)}$ and the diagrams in  $\Omega_{\{1\},\{s+r\}}$  look like

      \begin{center}
      \includegraphics{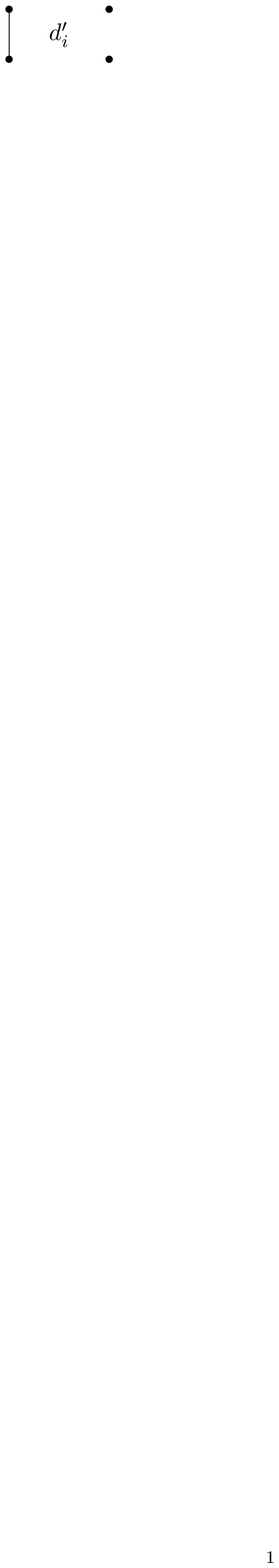}
      \end{center}

      The diagrams in $\Omega_{\{1\}, \{s+r\}}$ are indexed inductively as follows:

     Suppose $d'_i, 1 \leq i \leq {_{(s+r-2)}}C_{(s-1)}$ is the $i^{\text{th}}$ diagram in $\Omega^{s+r-2, s-1}$, and which is the the subdiagram of $d_i$  lying inbetween the first and last connected component will be the $i^{\text{th}}$ diagram in $\Omega_{\{1\}, \{s+r\}}$ inductively.

 \item[(ii)] For $1 \leq i \leq {_{(s+r-2)}}C_{(s-1)},$ let $d_i \in \Omega_{\{1\}, \{s+r\}}$ then $d_{i^{\ast}}$ is the diagram same as $d_i$ except they differ at the first and last connected component of $d_i$ where the first connected component of $d_{i^{\ast}}$ is a horizontal edge and the last connected component is a through class. Collection of such diagrams is denoted by $\Omega_{\{s+r\},\{1\}}.$ The diagram $d_{i^{\ast}}$ will look like

      \begin{center}
      \includegraphics{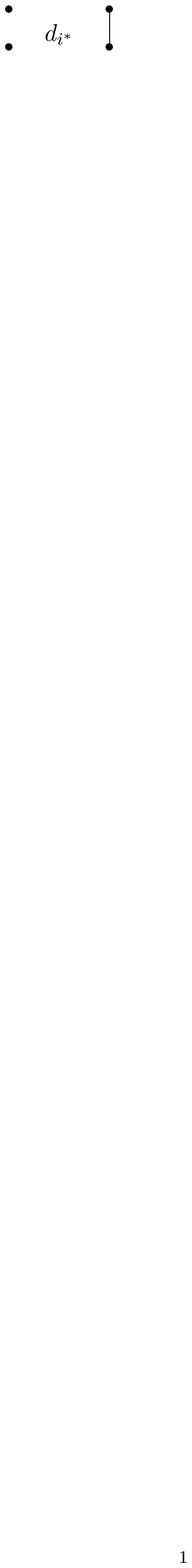}
      \end{center}
\item[(iii)] For $1 \leq i \leq {_{(s+r-2)}C_{(s-2)}},$  the collection of diagrams whose first and last connected components are through classes is denoted by $\Omega_{\{1, s+r\}, \{ \ \}}.$  The diagrams in $ \Omega_{\{1, s+r \}, \{ \ \}}$ will look like

      \begin{center}
      \includegraphics{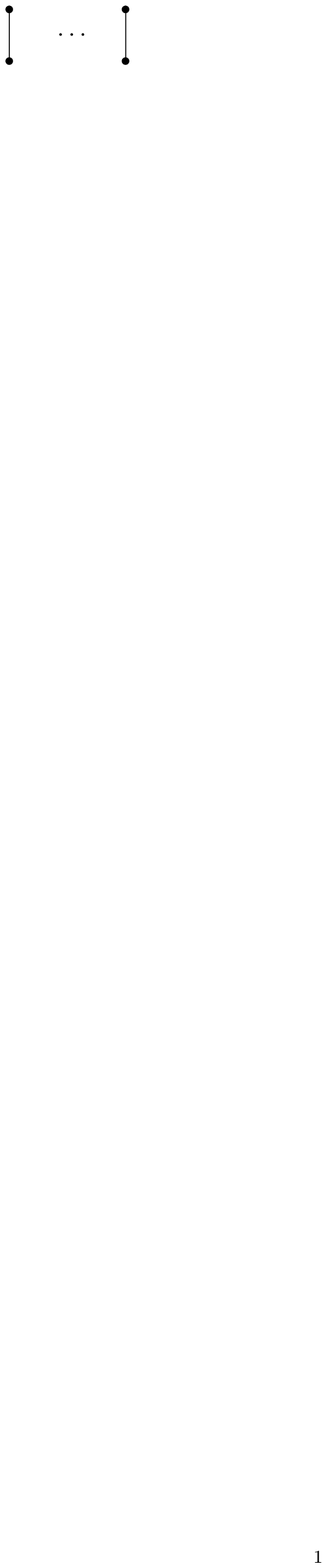}
      \end{center}
\item[(iv)] For $1 \leq i \leq {_{(s+r-2)}}C_{s},$ the collection of all diagrams whose first and last connected components are horizontal edges is denoted by $\Omega_{\{ \ \}, \{1, s+r\}}.$ The diagrams in $\Omega_{\{ \ \}, \{1, s+r\}}$ will look like
   \begin{center}
      \includegraphics{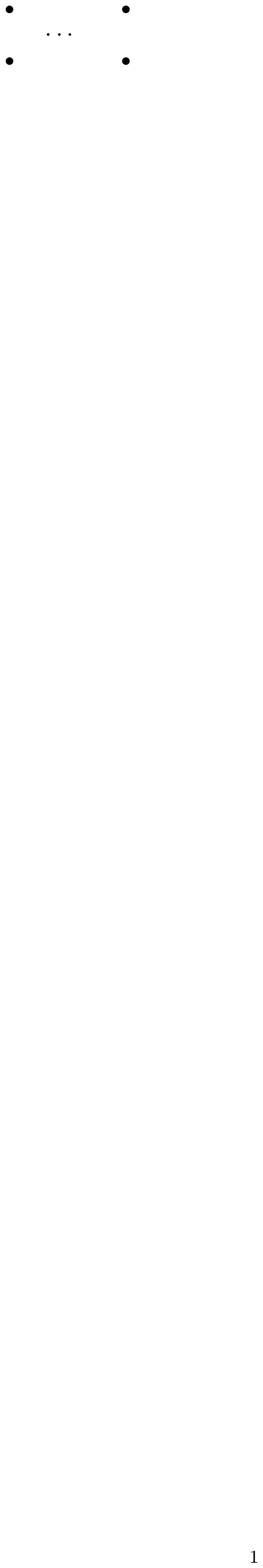}
      \end{center}
\end{enumerate}

\begin{lem}\label{L3.14}
Let $d_i, d_j \in \Omega_{\{m\}, \{f\}}$, $d_{i^{\ast}}, d_{j^{\ast}} \in \Omega_{\{f\}, \{m\}}$ and $a_{i j} = x_{\text{min} \{r,s\} - t}$ then
\begin{enumerate}
  \item[(i)]   $a_{i j} = a_{i^{\ast} j^{\ast}}.$
  \item[(ii)]  $a_{i^{\ast} j} = x_{\text{min} \{r, s\}-(t+1)}.$
\end{enumerate}
\end{lem}
\begin{proof}

\NI \textbf{Proof of (i):}
Let $d_i, d_j \in \Omega_{\{m\}, \{f\}}$ and $a_{i j} = x_{\text{min}\{r,s\}-t}$ where  $t$ denotes the number of horizontal edges in $d_i$ which are replaced by through classes in $d_j$ and vice versa. From the definition of $\Omega_{\{m\}, \{f\}},$ we know that the first and last connected components of $d_i$ and $d_j$ are the same.

By the definition of $\Omega_{\{f\}, \{m\}},$ we know that for every $d_i \in \Omega_{\{1\}, \{s+r\}}$ there exists a unique $d_{i^{\ast}} \in \Omega_{\{f\}, \{m\}}$ which differs only in the first and last connected component. Also, the first and last connected components of $d_{i^{\ast}}$ and $d_{j^{\ast}}$ are the same.

Thus, $a_{i j} = a_{i^{\ast} j^{\ast}}.$

\NI \textbf{Proof of (ii):} By the definition of $\Omega_{\{m\}, \{f\}},$ the $t$ number of horizontal edges which are replaced by the through classes and vice versa  lies in between the first and last connected components of $d_i$ and $d_j$. Also,the diagram $d_{i^{\ast}} \in \Omega_{\{f\}, \{m\}}$ is same as the diagram $d_i \in \Omega_{\{m\}, \{f\}}$ except the first and last connected component.

Therefore, including the first and last connected component there will be $t+1$ number of horizontal edges in $d_{i^{\ast}}$ which are replaced by through classes in $d_j$ and vice versa.

Thus, $a_{i^{\ast}, j} = x_{\text{min}\{s,r\}-(t+1)}.$
\end{proof}

\begin{lem}\label{L3.15}
 Let $d_i \in \Omega_{\{m\}, \{f\}}$ and $d_j \in \Omega_{\{m, f\}, \{ \}} \text{ or } \Omega_{\{ \ \}, \{m, f\}}$ then $$a_{i j} = a_{i^{\ast} j},\ \ \ \ \ 1 \leq i \leq {_{(s+r-2)}}C_{(s-1)}.$$

\end{lem}

\begin{proof}
\NI Let $a_{ij} = x_{\text{min}\{r,s\}-t}$ where $t$ denotes the number  horizontal edges in $d_i$ which is replaced by through classes and vice versa. From the definition of $\Omega_{\{m\}, \{f\}}$, we know that the diagram $d_{i^{\ast}}$ is same as the diagram $d_i \in \Omega_{\{m\}, \{f\}}$ except  at the first and last connected components where the first(last) connected component of $d_i$ is a through class(horizontal edge) but the first(last) connected component of $d_{i^{\ast}}$ is a horizontal edge(through class).

$a_{i^{\ast}, j}$ denotes the entry corresponding to the product of diagrams $d_{i^{\ast}}$ and $d_j$ then by the definition of $\Omega_{\{m\}, \{f\}}$ and $\Omega_{\{f\}, \{m\}}$ we have, $$a_{i j} = a_{i^{\ast} j},\ \ \ \ \ 1 \leq i \leq {_{(s+r-2)}}C_{(s-1)}.$$

\end{proof}

\NI We shall now prove the main theorem of this paper using the arrangement of diagrams given above and induction.

\NI \textbf{Proof of Main theorem:} The proof of the main theorem is using induction on the size of the matrix.

\NI \textbf{Case (i):} Let $s = 1, r=1$ and
\begin{center}
\includegraphics{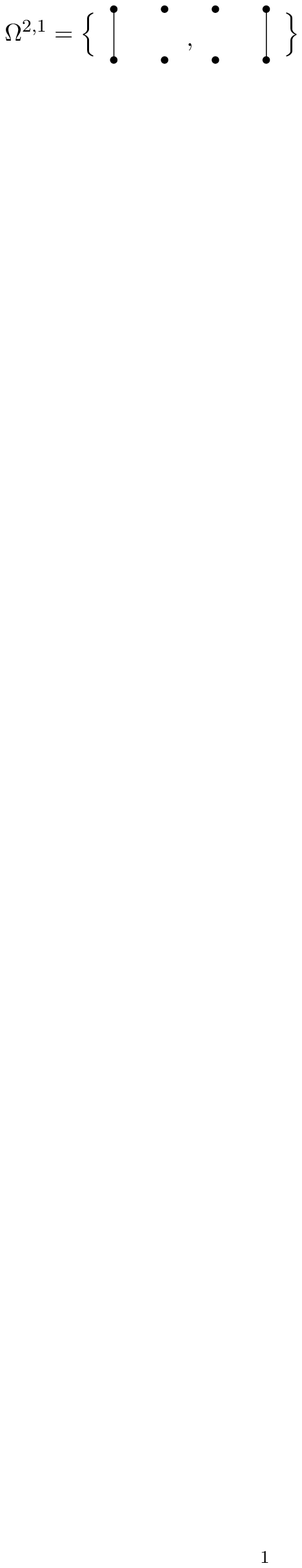}
\end{center}
By Definition \ref{D3.5}, the entries of the symmetric diagram matrix $A^{2,1}$ are $\{x_0, x_1\}$ where
\begin{center}
$A^{2,1} = \left(
             \begin{array}{cc}
               x_1 & x_0 \\
               x_0 & x_1 \\
             \end{array}
           \right)
$
\end{center}
Apply the following row operation and column operation on $A^{2,1}:$

\centerline{$R_d \leftrightarrow R_d - R_{d^{\ast}}, C_{d^{\ast}} \leftrightarrow C_{d^{\ast}} + C_d $ }
\NI where the diagrams $d$ and $d^{\ast}$ differ only at the first and last connected component. The first connected component of $d(d^{\ast})$ is through class(horizontal edge) and the last connected component of $d(d^{\ast})$ is horizontal edge(through class).

The reduced matrix is
\begin{center}
$\left(
   \begin{array}{cc}
     x_1-x_0 & 0 \\
     x_0 & x_1+x_0 \\
   \end{array}
 \right).
$
\end{center}
Thus, the eigenvalues of the symmetric diagram matrix $A^{2,1}$ are $x_1-x_0$ and $x_1+x_0.$

\NI \textbf{Case (ii):} Let $s = 2, r = 2$ and

\begin{center}
\includegraphics{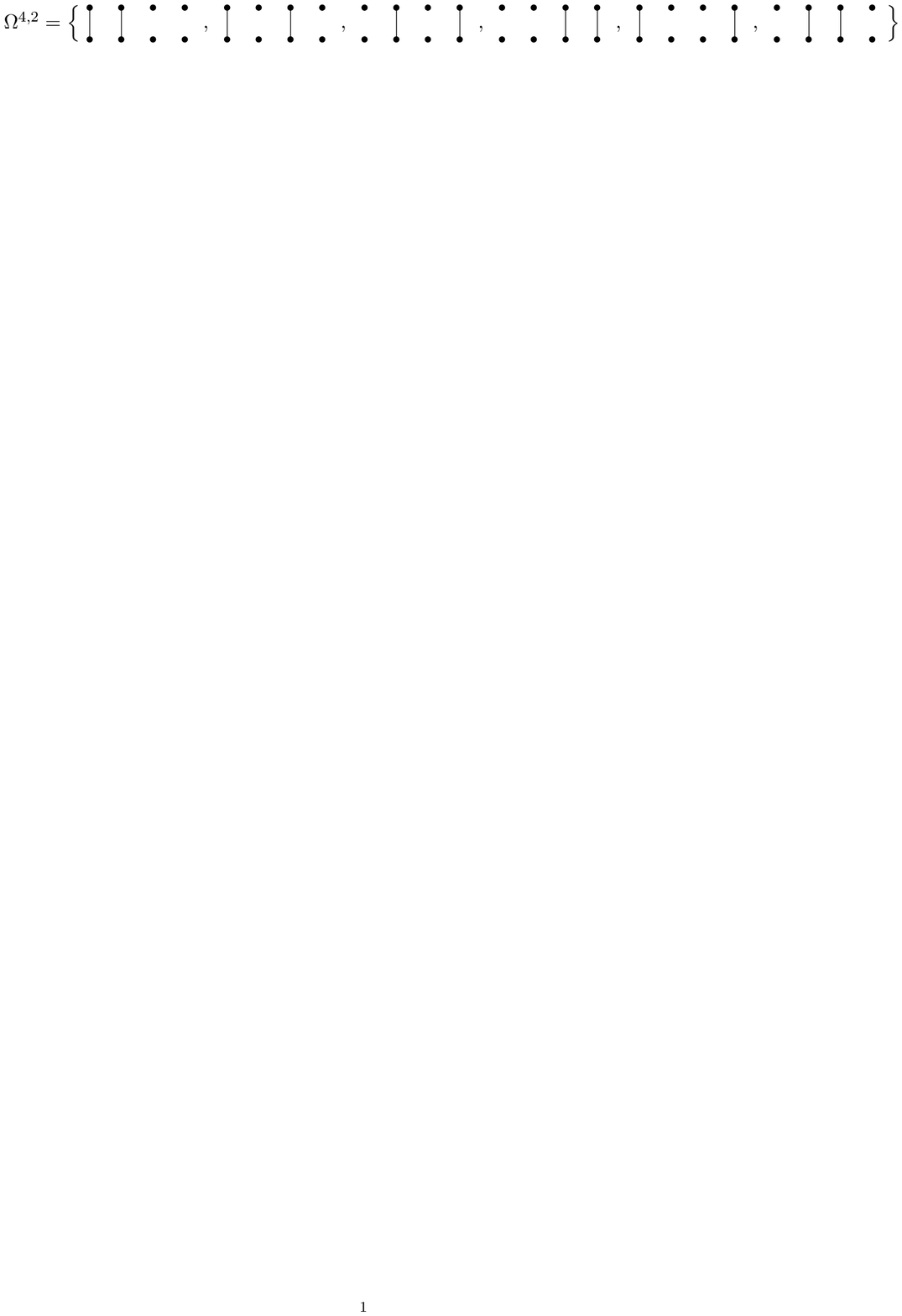}
\end{center}
\NI By Definition \ref{D3.5}, the entries of the symmetric diagram matrix $A^{4,2}$ are $\{x_0, x_1, x_2\}$ where

\begin{center}
$A^{4,2} = \left(
            \begin{array}{cccccc}
             x_2  & x_1 & x_1 & x_0 & x_1 & x_1 \\
             x_1  & x_2 & x_0 & x_1 & x_1 & x_1 \\
             x_1  & x_0 & x_2 & x_1 & x_1 & x_1 \\
             x_0  & x_1 & x_1 & x_2 & x_1 & x_1 \\
             x_1  & x_1 & x_1 & x_1 & x_2 & x_0 \\
             x_1  & x_1 & x_1 & x_1 & x_0 & x_2 \\
            \end{array}
          \right)$
\end{center}

Apply the following row and column operation on $A^{4, 2}$

\centerline{$R_d \leftrightarrow R_d - R_{d^{\ast}}, C_{d^{\ast}} \leftrightarrow C_{d^{\ast}} + C_d $}
\NI for all $d \in \Omega_{\{1\}, \{4\}}$ and $d^{\ast} \in \Omega_{\{4\}, \{1\}}$ which differs only at the first and last connected component as in Definition \ref{D3.12}. Then the reduced matrix is as follows:

\begin{center}
$\left(
  \begin{array}{cc}
    A_{1,1} & 0 \\
   \ast  & A_{1,2} \\
  \end{array}
\right)$
\end{center}
\NI where $A_{1,1} = \left(
                             \begin{array}{cc}
                               y_1 & y_0 \\
                               y_0 & y_1 \\
                             \end{array}
                           \right)
$ is a symmetric diagram matrix of size $2$ with $y_1 = x_2-x_1$ and $y_0 = x_1-x_0$ and
$A_{1,2} = \left(
                   \begin{array}{cccc}
                    x_2+x_1  & x_1+x_0 & x_1 & x_1 \\
                     x_1+x_0 & x_2+x_1 & x_1 & x_1 \\
                     2x_1 & 2x_1 & x_2 & x_0 \\
                     2x_1 & 2x_1 & x_0 & x_2 \\
                   \end{array}
                 \right)
$

The diagrams corresponding to the entries of the symmetric diagram matrix $A_{1, 1}$  belong to $\Omega_{\{1\}, \{4\}}$ and the diagrams corresponding to the entries of the matrix $A_{1, 2}$ belong to $\Omega_{\{4\}, \{1\}}, \Omega_{\{1,4\}, \{ \}}$ and $\Omega_{\{ \}, \{1, 4\}}.$

\NI Using induction, the eigenvalues of the symmetric diagram matrix $A^{4,2}_{1,1}$ are $y_1-y_0$ and $y_1 + y_0.$ i.e., $x_2 + x_0 - 2x_1$ and $x_2 - x_0.$

\NI Apply the following row and column operation on the matrix $A_{1,2}:$

\centerline{$R_d \leftrightarrow R_d - R_{d^{\ast}}, C_{d^{\ast}} \leftrightarrow C_{d^{\ast}} + C_d$}
\NI for the diagrams $d \in \Omega_{\{2,4\}, \{1, 3\}}$ and $d^{\ast} \in \Omega_{\{3, 4\}, \{1,2\}}.$

Thus, the reduced matrix is as follows:

\begin{center}
$\left(
   \begin{array}{cc}
     A_{2,1} & 0 \\
     \ast & A_{2,2} \\
   \end{array}
 \right)
$
\end{center}
\NI where $A_{2,1} = \left(x_2-x_0 \right)$ is a matrix of size $1$ and $A_{2,2} = \left(
             \begin{array}{ccc}
               x_2+2x_1+x_0 & x_1 & x_1 \\
               4x_1 & x_2 & x_0 \\
               4x_1 & x_0 & x_2 \\
             \end{array}
           \right).
$

\NI The diagram corresponding to the entry of the matrix $A_{2, 1}$ belong to $\Omega_{\{4,2\}, \{1,3\}}$ and the diagrams corresponding to the entries of the matrix $A_{2, 2}$ belong to $\Omega_{\{4,3\}, \{1,2\}}, \Omega_{\{1, 4\}, \{2, 3\}}$ and $\Omega_{\{2, 3\}, \{1, 4\}}.$

\NI The eigenvalue of the  matrix $A_{2,1}$ is $x_2 - x_0.$

\NI Apply the following row and column operations  on the matrix $A_{2, 2}:$ $$C_1 \leftrightarrow C_1 + C_2 + C_3, \ \ R_2 \leftrightarrow R_2 - R_1, \ \ R_3 \leftrightarrow R_3 - R_1$$
Thus, the reduced matrix is as follows:

\begin{center}
$\left(
   \begin{array}{cc}
     A_{3, 1} & \ast \\
     0 & A_{3,2} \\
   \end{array}
 \right)
$
\end{center}

\NI where $A_{3, 1} = \left( x_2+4x_1+x_0 \right)$ and $A_{3, 2} = \left(
                                                                     \begin{array}{cc}
                                                                       x_2-x_1 & x_0-x_1 \\
                                                                        x_0-x_1& x_2-x_1 \\
                                                                     \end{array}
                                                                   \right)
$ which is same as $A_{1, 1}.$

Therefore, the eigenvalues of the symmetric diagram matrix $A^{s+r, s}$ are $x_2 - x_0, x_2+4x_1+x_0$ and $x_2 -2x_1+x_0.$

In general, we shall do this for the diagrams in $\Omega^{s+r, s}$ and shall find the eigenvalues of the symmetric diagram matrix $A^{s+r, s}$ using induction.

\NI \textbf{Step I:}

We shall split $\Omega^{s+r, s}$ into four subsets as follows:

\centerline{$\Omega^{s+r, s} = \Omega_{\{1\}, \{s+r\}} \cup \Omega_{\{s+r\}, \{1\}} \cup \Omega_{\{1, s+r\}, \{\}} \cup \Omega_{\{\}, \{1, s+r\}}$} such that for all $d \in \Omega_{\{1\}, \{s+r\}}$ there exist a unique $d^{\ast} \in \Omega_{\{s+r\}, \{1\}}$ which differs only at the first and $(s+r)^{\text{th}}$ connected components where the first connected component of $d(d^{\ast})$ is through class (horizontal edge) and the $(s+r)^{\text{th}}$ connected component of $d(d^{\ast})$ is horizontal edge (through class).

Clearly,

$|\Omega_{\{1\}, \{s+r\}}| = {_{(s+r-2)}}C_{s-1} = |\Omega_{\{s+r\}, \{1\}}|, |\Omega_{\{1, s+r\}, \{ \}}| = {_{(s+r-2)}}C_{s-2}$ and $|\Omega_{\{ \}, \{ 1, s+r\}}| = {_{(s+r-2)}}C_s.$

Apply the following row and column operations on the symmetric diagram matrix $A^{s+r, s}$ of size ${_{(s+r)}}C_s.$

\centerline{$R_{d_i} \leftrightarrow R_{d_i} - R_{d_{i^{\ast}}} \ \ \ \ \ \forall d_i \in \Omega_{\{1\}, \{s+r\}}$}
\NI and

\centerline{$C_{d_{i^{\ast}}} \leftrightarrow C_{d_{i^{\ast}}} + C_{d_i} \ \ \ \ \ \ \forall d_{i^{\ast}} \in \Omega_{\{s+r\}, \{1\}}.$}

\NI Using lemmas \ref{L3.14} and \ref{L3.15}, the reduced matrix is as follows:

\begin{center}
$\left(
   \begin{array}{cc}
     A_{1,1} & 0 \\
     \ast & A_{1,2} \\
   \end{array}
 \right)
$
\end{center}

\NI where $A_{1,1}$  is a symmetric diagram matrix of size ${_{(s+r-2)}}C_{(s-1)}$ with entries $\{y_0, y_1, \cdots, y_{\text{min}\{s-1, r-1\}}\}$ and $y_i = x_{i+1} - x_i.$

\NI The diagrams corresponding to the entries of the symmetric diagram matrix $A_{1, 1}$ belong to $\Omega_{\{1\}, \{s+r\}}.$

\NI Using induction the eigenvalues of the symmetric diagram matrix $A_{1,1}$ of size ${_{(s+r-2)}}C_{(s-1)}$ are given as

\begin{equation}\label{E3.2}
\underset{t=0}{\overset{\text{min }\{s-1, r-1\}}{\sum}} \left\{ \underset{j=0}{\overset{l-1}{\sum}} (-1)^j \ {_{(l-1)}}C_j \ {_{[(s-1)-(l-1)]}}C_{(t-j)} \ {_{[(r-1)-(l-1)]}}C_{(t-j)}\right\} y_{\text{min} \{s-1, r-1\}-t}
\end{equation}
for all $ 0 \leq l-1 \leq \text{ min } \{s-1, r-1\}.$

Substitute $y_{\text{min} \{s-1, r-1\} -t} = x_{\text{min}\{s-1, r-1\}-t+1} - x_{\text{min}\{s-1, r-1\}-t}$ in expression (\ref{E3.2}) we get,
\begin{eqnarray*}
   &=& \underset{t=0}{\overset{\text{min }\{s-1, r-1\}}{\sum}} \left\{ \underset{j=0}{\overset{l-1}{\sum}} (-1)^j \ {_{(l-1)}}C_j \ {_{[(s-1)-(l-1)]}}C_{(t-j)} \ {_{[(r-1)-(l-1)]}}C_{(t-j)}\right\} \\
    & & \hspace{8cm} \left\{ x_{\text{min}\{s-1,r-1\}-t+1} - x_{\text{min}\{s-1,r-1\}-t} \right\} \\
    & =& \underset{t=0}{\overset{\text{min}\{s-1,r-1\}}{\sum}} \left\{ \underset{j=0}{\overset{l-1}{\sum}} (-1)^j \ {_{(l-1)}}C_j \ {_{(s-l)}}C_{(t-j)} \ {_{(r-l)}}C_{(t-j)} - \underset{j=0}{\overset{l-1}{\sum}} (-1)^j \ {_{(l-1)}}C_j \ {_{(s-l)}}C_{(t-1-j)} \ {_{(r-l)}}C_{(t-1-j)}\right\}\\
   & &  \hspace{13cm} x_{\text{min}\{s-1, r-1\}-t+1}\\
   & & \hspace{6cm}(\text{By collecting the coefficients } x_{\text{min}\{s-1, r-1\}-t+1})\\
    &=& \underset{t=0}{\overset{\text{min}\{s-1,r-1\}}{\sum}} \left\{ \underset{j=0}{\overset{l-1}{\sum}} (-1)^j \ {_{(l-1)}}C_j \ {_{(s-l)}}C_{(t-j)} \ {_{(r-l)}}C_{(t-j)} - \underset{j'=1}{\overset{l}{\sum}} (-1)^{j'-1} \ {_{(l-1)}}C_{(j'-1)} \ {_{(s-l)}}C_{(t-j')} \ {_{(r-l)}}C_{(t-j')}\right\}\\
   & &  \hspace{13cm} x_{\text{min}\{s-1, r-1\}-t+1}\\
   &=&  \underset{t=0}{\overset{\text{min}\{s-1,r-1\}}{\sum}} \left\{ {_{(s-l)}}C_t \ {_{(r-l)}}C_t + \underset{j=0}{\overset{l-1}{\sum}} (-1)^j \ {_{(l-1)}}C_j \ {_{(s-l)}}C_{(t-j)} \ {_{(r-l)}}C_{(t-j)} + (-1)^l \ {_{(s-l)}}C_{(t-l)} \ {_{(r-l)}}C_{(t-l)}\right\} \\
   & &  \hspace{13cm} x_{\text{min}\{s-1, r-1\}-t+1}
      \end{eqnarray*}

    \begin{eqnarray*}
         &=& \underset{t=0}{\overset{\text{min}\{s-1,r-1\}}{\sum}} \left\{ \underset{j=0}{\overset{l}{\sum}} (-1)^j \ {_{(l-1)}}C_j \ {_{(s-l)}}C_{(t-j)} \ {_{(r-l)}}C_{(t-j)}\right\} x_{\text{min}\{s-1, r-1\}-t+1} \\
    & & \hspace{9cm} \text{ for all } \ 0 \leq l-1 \leq \text{min}\{s-1, r-1\}\\
   &=& \underset{t=0}{\overset{\text{min}\{s-1,r-1\}}{\sum}} \left\{ \underset{j=0}{\overset{l}{\sum}} (-1)^j \ {_{(l-1)}}C_j \ {_{(s-l)}}C_{(t-j)} \ {_{(r-l)}}C_{(t-j)}\right\} x_{\text{min}\{s, r\}-t} \ \ \ \forall \ 1 \leq l \leq \text{min}\{s, r\}
\end{eqnarray*}

\NI Thus, the eigenvalues of the  symmetric diagram matrix $A_{1, 1}$ of size ${_{(s+r-4)}}C_{(s-2)}$ are given by

\begin{equation}\label{E}
 \underset{t=0}{\overset{\text{min}\{s, r\}}{\sum}} \left\{ \underset{j=0}{\overset{l}{\sum}} (-1)^j \ {_{(l-1)}}C_j \ {_{(s-l)}}C_{(t-j)} \ {_{(r-l)}}C_{(t-j)}\right\} x_{\text{min}\{s, r\}-t} \ \ \ \forall \ 1 \leq l \leq \text{min}\{s, r\}.
\end{equation}

\NI Therefore, the eigenvalues of the symmetric diagram matrix $A^{s+r,s}$ of size ${_{(s+r-2)}}C_{(s-1)}$ are

\begin{equation}\label{E3.4}
 \underset{t=0}{\overset{\text{min}\{s, r\}}{\sum}} \left\{ \underset{j=0}{\overset{l}{\sum}} (-1)^j \ {_{(l-1)}}C_j \ {_{(s-l)}}C_{(t-j)} \ {_{(r-l)}}C_{(t-j)}\right\} x_{\text{min}\{s, r\}-t} \ \ \ \forall \ 1 \leq l \leq \text{min}\{s, r\}.
\end{equation}

 The number of distinct eigenvalues of the symmetric diagram matrix $A^{s+r, s}$ so far computed which is given in (\ref{E3.4}) are min $\{s, r\}.$

    We shall now prove that the eigenvalues of the submatrix $A_{1, 2}$ are also as given in expression (\ref{E}).

\NI \textbf{Step 2:} The diagrams corresponding to the entries of the matrix $A_{1,2}$ belong to $\Omega_{I_i, J_i}$ for $1 \leq i \leq 3$ where  $I_1 = \{s+r\}, I_2 = \{1, s+r\}, I_3 = \{ \ \}, J_1 = \{1\}, J_2 = \{ \ \}$ and $J_3 =  \{1, s+r\}.$

We shall apply the following row and column operations on $A_{1, 2}:$

\centerline{$ R_d \leftrightarrow R_d - R{d^{\ast}} \ \ C_{d^{\ast}} \leftrightarrow C_{d^{\ast}} + C_d \ \ \forall \ d \in \Omega_{I_i \cup \{2\}, J_i \cup \{s+r-1\}} $ and $d^{\ast} \in \Omega_{I_i \cup \{s+r-1\}, J_i \cup \{2\}}$}

\NI for all $1 \leq i \leq 3.$

Using Lemmas \ref{L3.14} and \ref{L3.15} and  applying suitable  row and column operations to the reduced matrix looks like,

\begin{center}
$\left(
   \begin{array}{cc}
     A_{2,1} & 0 \\
     \ast & A_{2,2} \\
   \end{array}
 \right).
$
\end{center}

The diagrams corresponding to the entries of the matrix $A_{2,1}$ belong to $\Omega_{I_i \cup \{2\}, J_i \cup \{s+r-1\}}$ for all $1 \leq i \leq 3.$ The size of the matrix $A_{2,1}$ is ${_{(s+r-4)}}C_{(s-2)} + {_{(s+r-4)}}C_{(s-3)} + {_{(s+r-4)}}C_{(s-1)} = \underset{j=1}{\overset{3}{\sum}} {_{(s+r-4)}}C_{(s-j)}.$

Now, we shall show that the eigenvalues of the matrix $A_{2, 1}$ belong to the collection of all eigenvalues of the matrix $A_{1, 1}$ obtained in Step 1.

We know that the diagrams corresponding to the entries of the symmetric diagram matrix $A_{1, 1}$ obtained in Step $1$ belong to $\Omega_{\{1\}, \{s+r\}}.$

Apply the following row and column operations on the symmetric diagram matrix $A_{1,1}:$

\centerline{$R_d \leftrightarrow R_d - R_{d^{\ast}}, \ \ C_{d^{\ast}} \leftrightarrow C_{d^{\ast}} + C_d$ }

\NI for all $ d \in \Omega_{\{1\} \cup \{ 2\}, \{s+r-1\} \cup \{ s+r\}}$ and $d^{\ast} \in \Omega_{\{1\} \cup\{ s+r-1\}, \{2\} \cup \{ s+r\}}.$

Using induction on the number of through classes the reduced matrix is as follows:

\begin{center}
$\left(
   \begin{array}{cc}
     B_{1, 1} & 0 \\
     \ast & B_{1, 2} \\
   \end{array}
 \right)
$
\end{center}
where the diagrams corresponding to the entries of the matrix $B_{1,1}$ belong to $\Omega_{\{1,2\}, \{s+r-1, s+r\}}$ and the diagrams corresponding to the entries of the matrix $B_{1, 2}$ belong to $\Omega_{\{1, s+r-1\}, \{2, s+r\}}, \Omega_{\{1, 2, s+r-1\}, \{s+r\}}$ and $ \Omega_{\{1\}, \{2, s+r-1, s+r\}}.$

The size of the matrix $B_{1, 1}$ is ${_{(s+r-4)}}C_{(s-2)}$ and the size of the matrix $B_{1, 2}$ is $\underset{j=1}{\overset{3}{\sum}} {_{(s+r-4)}}C_{(s-j)}.$

\NI Since, $B_{1, 1}$ and $B_{1,2}$ are the submatrices of the symmetric diagram matrix  $A_{1,1}$, by induction the eigenvalues of the matrices $B_{1, 1}$ and $B_{1, 2}$ belong to the collection of all eigenvalues given in expression \ref{E}.

\NI It is clear from the definition of $\Omega_{I, J}$ defined in Definition \ref{D3.12} that the submatrices $B_{1, 2}$ and $A_{2, 1}$ are the same. Hence the eigenvalues of the matrix $B_{1, 2}$ and $A_{2, 1}$ are the same.

\NI The diagrams corresponding to the entries of the matrix $A_{2, 2}$ belong to $\Omega_{I_i \cup \{s+r-1\}, J_i \cup \{2\}}, \Omega_{I_i \cup \{2, s+r-1\}, J_i }$ and $\Omega_{I_i, J_i \cup \{2, s+r-1\}}$ for $1 \leq i \leq 3.$

\textbf{Step j:} In general, the diagrams corresponding to the entries of the matrix $A_{j-1, 2}$ obtained in step $j-1$ belong to $\Omega_{I_i \cup  \{s+r-(j-2)\}, J_i \cup \{j-1\}}, \Omega_{I_i \cup \{j-1, s+r-(j-2)\}, J_i}$ and $\Omega_{I_i, J_i \cup \{j-1, s+r-(j-2)\}}$ with $I_i \cap J_i = \emptyset$ for all $1 \leq i \leq 3^{j-2}.$

Apply the following row and column operations on $A_{j-1, 2}:$

\centerline{$R_d \leftrightarrow R_d - R_{d^{\ast}}, \ \ C_{d^{\ast}} \leftrightarrow C_{d^{\ast}} + C_d$ }
\NI for all $d \in \Omega_{I_i \cup \{j\}, J_i \cup \{s+r-(j-1)\}}$ and $d^{\ast} \in \Omega_{I_i \cup \{s+r-(j-1)\}, J_i \cup \{j\}}.$

Using Lemmas \ref{L3.14} and \ref{L3.15} and interchanging the rows and columns suitably the reduced matrix is as follows:

\begin{center}
$\left(
   \begin{array}{cc}
     A_{j, 1} & 0 \\
     \ast & A_{j, 2} \\
   \end{array}
 \right)
$
\end{center}

\NI The diagrams corresponding to the entries of the matrix $A_{j,1}$ belong to $\Omega_{I_i \cup \{j\}, J_i \cup \{s+r-(j-1)\}}$ for all $1 \leq i \leq 3^{j-1}.$ The size of the matrix $A_{j,1}$ is $ \underset{i=1}{\overset{3^{j-1}}{\sum}} {_{(s+r-(|I_i| + |J_i|))}}C_{(s-|I_i|)}.$

\NI Now, we shall show that the eigenvalues of the matrix $A_{j, 1}$ and the eigenvalues of the matrix $B_{j-1, 2}$ obtained from $B_{1,2}$ inductively as in Step 2 are the same.

\NI Using induction, we know that the diagrams corresponding to the entries of the symmetric diagram matrix $B_{j-2, 2}$ obtained in Step $j-2$ belong to $\Omega_{I'_i \cup \{s+r-(j-3)\}, J'_i \cup \{j-2\}}, \Omega_{I'_i \cup \{j-2,  s+r-(j-3)\}, J'_i}$ and $\Omega_{I'_i, J'_i \cup \{j-2, s+r-(j-3)\}}$ for all $1 \leq i \leq 3^{j-3}.$

\NI Apply the following row and column operations on the symmetric diagram matrix $B_{j-2,2}:$

\centerline{$R_d \leftrightarrow R_d - R_{d^{\ast}}, \ \ C_{d^{\ast}} \leftrightarrow C_{d^{\ast}} + C_d$ }

\NI for all $ d \in \Omega_{I'_i \cup \{ j-1\}, J'_i \cup \{ s+r-(j-2)\}}$ and $d^{\ast} \in \Omega_{I'_i \cup\{ s+r-(j-2)\}, J'_i \cup \{ j-1\}}, 1 \leq i \leq 3^{j-2}.$

\NI Using induction on the number of through classes, Lemma \ref{L3.14}, \ref{L3.15} and interchanging rows and column suitably the reduced matrix is as follows:

\begin{center}
$\left(
   \begin{array}{cc}
     B_{j-1, 1} & 0 \\
     \ast & B_{j-1, 2} \\
   \end{array}
 \right)
$
\end{center}
where the diagrams corresponding to the entries of the matrix $B_{j-1,1}$ belong to $\Omega_{I'_i \cup \{j-1\}, J'_i \cup \{s+r-(j-2)\}}$ and the diagrams corresponding to the entries of the matrix $B_{j-1, 2}$ belong to $\Omega_{I'_i \cup \{s+r-(j-2)\}, J'_i \cup \{j-1\}}, \\ \Omega_{I'_i \cup \{j-1, s+r-(j-2)\}, J'_i }$ and $ \Omega_{I'_i, J'_i \cup \{j-1, s+r-(j-2)\}}$ for all $1 \leq i \leq 3^{j-2}.$

The size of the matrix  $B_{j-1, 2}$ is $\underset{m=1}{\overset{3^{j-1}}{\sum}} {_{(s+r-|I_i| - |J_i|)}}C_{(s-|I_i|)}.$

\NI Since, $B_{j-1, 1}$ and $B_{j-1,2}$ are the submatrices of the symmetric diagram matrix  $A_{1,1}$, the eigenvalues of the matrices $B_{j-1, 1}$ and $B_{j-1, 2}$ belong to the collection of all eigenvalues given in expression \ref{E}.

\NI It is clear from the definition of $\Omega_{I, J}$ defined in Definition \ref{D3.12} that the submatrices $B_{j-1, 2}$ and $A_{j, 2}$ are the same. Hence, eigenvalues of the matrix $B_{j-1, 2}$ and $A_{j, 1}$ are the same.

\NI The diagrams corresponding to the entries of the matrix $A_{j, 2}$ belong to $\Omega_{I_i \cup \{s+r-(j-1)\}, J_i \cup \{j\}}, \\ \Omega_{I_i \cup \{j, s+r-(j-1)\}, J_i }$ and $\Omega_{I_i, J_i \cup \{j, s+r-(j-1)\}}$ for $1 \leq i \leq 3^{j-1}.$

\begin{note}
Some of $I_i$ and $J_i$ may be empty. When the set is empty we leave that set and continue the process. Also, we always consider the matrix with coefficient $1$ for the highest order in the determinant.
\end{note}

This process is continued till step $p$ if $s+r = 2p$ and $p-1$ if $s+r = 2p+1.$

If $s+r = 2p+1$ then we apply the following row and column operations on $A_{p-1, 2}:$

\centerline{$R_d \leftrightarrow R_d - R_{d^{\ast}}$ and $C_{d^{\ast}} \leftrightarrow C_{d^{\ast}} + C_d$}
\NI for all $d \in  \Omega_{I_i \cup \{p\}, J_i \cup \{p+1\}}$ and $d \in \Omega_{I_i \cup \{p+1\}, J_i \cup \{p+2\}}, d^{\ast} \in  \Omega_{I_i \cup \{p+1\}, J_i \cup \{p\}}$ and $d^{\ast} \in \Omega_{I_i \cup \{p+2\}, J_i \cup \{p+1\}}, \\ 1 \leq i \leq 3^{p-1}.$

In both the cases (i.e., $s+r = 2p$ and $s+r = 2p+1$), using Lemmas \ref{L3.14}, \ref{L3.15} and by interchanging rows and columns suitably the reduced matrix is as follows:

\begin{center}
$\left(
   \begin{array}{cc}
     A_{p, 1} & 0  \\
     \ast & A_{p, 2} \\
   \end{array}
 \right).
$
\end{center}

The diagrams corresponding to the entries of the matrix $A_{p, 2}$ belong to $\Omega_{I_i, J_i}$ for all $1 \leq i \leq 3^p.$

Finally, we shall apply the following row and column operations on the matrix $A_{p, 2}:$

\begin{enumerate}
  \item[(i)] Fix a $d \in \Omega_{\{s+r\}, \{1\}} \setminus \underset{i=1}{\overset{3^p}{\cup}} \Omega_{I_i, J_i}.$

      \centerline{$C_d \leftrightarrow C_d + \underset{d'}{\sum} C_{d'}$}

  where $d' \in \Omega_{\{s+r\}, \{1\}} \setminus \underset{i=1}{\overset{3^p}{\cup}} \Omega_{I_i, J_i}.$

  \centerline{$R_{d'} \leftrightarrow R_{d'} - R_d$ for all $d' \in \Omega_{\{s+r\}, \{1\}} \setminus \underset{i=1}{\overset{3^p}{\cup}} \Omega_{I_i, J_i}.$}

  \item[(ii)] Fix a $d \in \Omega_{\{1, s+r\}, \{ \ \}} \setminus \underset{i=1}{\overset{3^p}{\cup}} \Omega_{I_i, J_i}.$

      \centerline{$C_d \leftrightarrow C_d + \underset{d'}{\sum} C_{d'}$}

  where $d' \in \Omega_{\{1,s+r\}, \{ \ \}} \setminus \underset{i=1}{\overset{3^p}{\cup}} \Omega_{I_i, J_i}.$

  \centerline{$R_{d'} \leftrightarrow R_{d'} - R_d$ for all $d' \in \Omega_{\{1, s+r\}, \{ \ \}} \setminus \underset{i=1}{\overset{3^p}{\cup}} \Omega_{I_i, J_i}.$}

  \item[(iii)]   Fix a $d \in \Omega_{\{ \ \}, \{1, s+r\}} \setminus \underset{i=1}{\overset{3^p}{\cup}} \Omega_{I_i, J_i}.$

      \centerline{$C_d \leftrightarrow C_d + \underset{d'}{\sum} C_{d'}$}

  where $d' \in \Omega_{\{ \ \}, \{1, s+r\}} \setminus \underset{i=1}{\overset{3^p}{\cup}} \Omega_{I_i, J_i}.$

  \centerline{$R_{d'} \leftrightarrow R_{d'} - R_d$ for all $d' \in \Omega_{\{ \ \}, \{1, s+r\}} \setminus \underset{i=1}{\overset{3^p}{\cup}} \Omega_{I_i, J_i}.$}

\end{enumerate}

Using induction, Lemma \ref{L3.14}, Lemma \ref{L3.15} and after applying suitable row and column operations on the reduced matrix, it becomes,

\begin{center}
$\left(
   \begin{array}{cc}
     A_{p+1, 1} & 0 \\
     \ast & A_{p+1, 2} \\
   \end{array}
 \right)
$
\end{center}
\NI where the size of the matrix $A_{p+1, 1}$ is $3$ and the matrix $A_{p+1, 2}$ is one among the submatrices obtained after applying row and column operations on the symmetric diagram matrix $A_{1, 1}$ inductively.

Thus, the eigenvalues of the matrix $A_{p+1, 2}$ belong to the collection of all eigenvalues given in expression \ref{E}.

Now, we are left out to find  the eigenvalues of the matrix $A_{p+1, 1}.$ For that we add all the entries to one column and subtract the corresponding row with other rows.

The reduced matrix is as follows:

\begin{center}
$\left(
   \begin{array}{cc}
     A' & \ast \\
     0 & A'' \\
   \end{array}
 \right)
$
\end{center}
where $A'$ is a $1 \times 1$ matrix and $A''$ is a $2 \times 2$ matrix.

Since, we have only performed addition on the columns the entry of the matrix $A'$ is the sum of the entries of the symmetric diagram matrix $A^{s+r, s}$ of size ${_{(s+r)}}C_{s}$ which is given by

$\underset{t=0}{\overset{\text{min}\{s, r\}}{\sum}} {_s}C_t \ {_r}C_t \ x_{\text{min} \{s, r\}-t}.$

\NI we perform the following row and column operations on the matrix $A'':$

\centerline{$R_1 \leftrightarrow R_1 + R_2$ and $C_2 \leftrightarrow C_2 - C_1$}

 Thus, the eigenvalues of the matrix $A''$ also belong to the collection of all eigenvalues given in expression \ref{E}.


Thus, we have computed all the $\text{min} \{s, r\} + 1$ number of distinct eigenvalues of the symmetric diagram matrix $A^{s+r, s}.$

Thus, the eigenvalues of the symmetric diagram matrix $A^{s+r, s}$ of size ${_{(s+r)}}C_{s}$ is given by

\centerline{$\underset{t=0}{\overset{\text{min}\{s, r\}}{\sum}} \left\{ \underset{j=0}{\overset{l}{\sum}} (-1)^j \ {_l}C_j \ {_{(s-l)}}C_{(t-j)} \ {_{(r-l)}}C_{(t-j)}\right\} x_{\text{min}\{s, r\} -t}$}

for all $0 \leq l \leq \text{min}\{s, r\}.$

\section{\textbf{Eigenvalues of Gram matrices for a class of diagram algebras}}
\subsection{Eigenvalues of the Gram matrices of the partition algebras  }
\textbf{\\}
In this section, we compute the eigenvalues of Gram matrices of partition algebras where the block submatrices of the Gram matrix are realized as the direct product of the symmetric diagram matrices.

\begin{thm}\label{T4.1}
\begin{enumerate}
\item[(i)]The set of all distinct eigenvalues of Gram matrix  $G_s$ of the partition algebra with entries $\{X_0, X_1, \cdots, X_{\text{min}\{s, r\}}\}$ are given by

\centerline{$ \underset{t=0}{\overset{\text{min}\{s, r\}}{\sum}} \left[ \underset{j=0}{\overset{l}{\sum}} (-1)^j \ {_l}C_j \ {_{(s-l)}}C_{(t-j)} \ {_{(r-l)}}C_{(t-j)}\right] X_{\text{min}\{s, r\}-t}$}

\NI for all $0 \leq l \leq \text{min}\{r, s\}$ and $0 \leq r \leq k-s$ where $X_{\text{min}\{s, r\}-t} = (-1)^t \ t! \ \underset{m=i}{\overset{r-1}{\prod}} [x-(s+l)], k$ and $s$ are fixed integers.

\item[(ii)]The eigenvalues of the Gram matrix $G_s$ are integers which is given by

\NI \hspace{-0.5cm} $\underset{t=0}{\overset{\text{min}\{s, r\}}{\sum}} \left[ \underset{j=0}{\overset{l}{\sum}} (-1)^j \ {_l}C_j \ {_{(s-l)}}C_{(t-j)} \ {_{(r-l)}}C_{(t-j)}\right] X_{\text{min}\{s, r\}-t} = \underset{i=0}{\overset{l-1}{\prod}} [x-(s-1+i)] \underset{j=0}{\overset{\text{min}\{s, r\}-l-1}{\prod}}[x-(2s+j)]$

\NI  $0 \leq l \leq \text{min}\{s, r\}.$
\end{enumerate}
\end{thm}

\begin{proof}
\NI \textbf{Proof of (i):}Let  $U^{R^d}_{R^d} \in \mathbb{J}_s^r$  such that $\phi'(R^d) = \lambda'$ where $\mathbb{J}_s^r$ and $\Omega_s^r$ are as in Notation \ref{N2.12} and Definition \ref{D2.11} respectively.

We shall draw a diagram  using the diagram $R^{d}$ with $s$ through classes as follows:
\begin{enumerate}
  \item[(i)] Draw $R^{d}$ denoted by $R^{d^+}$ in the top row and a copy of $R^{d}$ denoted by $R^{d^-}$ in the bottom row.
  \item[(ii)] Among the $s+r$ connected components in the top row choose $s$ connected components and join each connected component in the top row with the respective connected component in the bottom row by vertical edges.
  \item[(iii)] Denote the collection of such diagrams as $\mathbb{J}^{R^d, r}_s.$  In particular $U^{R^d}_{R^d} \in \Omega^{s+r, r}_{R^d}.$
\end{enumerate}

\NI It is clear that the number of such diagrams with $s$ through classes is ${_{(s+r)}}C_s$ and $\mathbb{J}^{r}_{s} = \underset{R^d}{\cup} \ \Omega^{s+r, r}_{R^d}$ where $\mathbb{J}^r_s$ is as in Notation \ref{N2.12}.

Suppose $U^{R^{d'}}_{R^{d'}} \in \mathbb{J}^{r'}_{s}$ then the entry in the Gram matrix corresponding to the product $U^{R^d}_{R^d} \cdot U^{R^{d'}}_{R^{d'}}$ is either  $ x^{r''}$ with $r'' < r$ or $0$ mod $\lambda'.$

\NI \textbf{Case (i):} If the entry corresponding to the product $U^{R^d}_{R^d} \cdot U^{R^{d'}}_{R^{d'}}$ in the Gram matrix $G_s$ is $x^{r''}$ with $r'' < r$ and while applying the column operations to eliminate the entries corresponding to the diagrams coarser than $U^{R^d}_{R^d}$ the entry $x^{r''}$ becomes zero by Lemma 3.20 in \cite{KP}.

\NI \textbf{Case (ii):} If the entry corresponding the product $U^{R^d}_{R^d} \cdot U^{R^{d'}}_{R^{d'}}$ in the Gram matrix $G_s$ is $0$ mod $\lambda'$ and it remains zero even after applying column operations to eliminate the entries corresponding to the diagrams coarser than $U^{R^d}_{R^d}$ by Lemma 3.19 and Lemma 3.23 in \cite{KP}.

Rearranging the diagrams $U^{R^d}_{R^d} \in \mathbb{J}_s^r$ in such a way that

\centerline{$A'_{r, r} = \underset{R^d}{\prod} A^{s+r, s}_{R^d}$}
\NI where $A^{s+r, s}_{R^d}$ is a symmetric diagram matrix of size ${_{(s+r)}}C_{s}$ with entries $\{X_0, X_1, \cdots, X_{\text{min}\{s, r\}}\}$ and the diagrams corresponding to the entries of the symmetric diagram matrix $A^{s+r, s}$ belong to $\Omega^{s+r, r}_{R^d}.$

\NI Also, the by Theorem \ref{T2.17}(c) $X_{\text{min}\{s, r\}-t} = (-1)^t \ t! \ \underset{m=t}{\overset{r-1}{\prod}} [x-(s+m)].$

\NI By theorem \ref{T3.9}, the eigenvalues of the symmetric diagram matrix $A^{s+r, s}_{R^d}$ of size ${_{(s+r)}}C_{s}$ are given by

\centerline{$ \underset{t=0}{\overset{\text{min}\{s, r\}}{\sum}} \left[ \underset{j=0}{\overset{l}{\sum}} (-1)^j \ {_l}C_j \ {_{(s-l)}}C_{(t-j)} \ {_{(r-l)}}C_{(t-j)}\right] X_{\text{min}\{s, r\}-t}$}

\NI for all $0 \leq l \leq \text{min}\{s, r\}.$

\NI Thus, the eigenvalues of the block submatrix $A'_{r, r}$ which is a direct product of symmetric diagram matrices $A^{s+r, s}_{R^d}$ are given by

\centerline{$ \underset{t=0}{\overset{\text{min}\{s, r\}}{\sum}} \left[ \underset{j=0}{\overset{l}{\sum}} (-1)^j \ {_l}C_j \ {_{(s-l)}}C_{(t-j)} \ {_{(r-l)}}C_{(t-j)}\right] X_{\text{min}\{s, r\}-t}$}

\NI for all $0 \leq l \leq \text{min}\{s, r\}.$

\NI Therefore, the eigenvalues of the Gram matrix $G_s$ are given by

\centerline{$ \underset{t=0}{\overset{\text{min}\{s, r\}}{\sum}} \left[ \underset{j=0}{\overset{l}{\sum}} (-1)^j \ {_l}C_j \ {_{(s-l)}}C_{(t-j)} \ {_{(r-l)}}C_{(t-j)}\right] X_{\text{min}\{s, r\}-t}$}

\NI for all $0 \leq l \leq \text{min}\{s, r\}$ and $0 \leq r \leq k-s.$

\NI \textbf{proof of (ii):} The proof of (ii) follows by  comparing the coefficients on both the sides.

\end{proof}

\subsection{Eigenvalues of Gram Matrices for Signed Partition Algebras and the algebra of $\mathbb{Z}_2$-relations:}
\textbf{\\}

\begin{thm}\label{T4.2}
\begin{enumerate}
\item[(i)] The set of all distinct eigenvalues of Gram matrices $G_{2s_1+s_2}$ of the algebra of $\mathbb{Z}_2$-relations with entries $\{X_0, X_1, \cdots, X_{\text{min}\{s_1, r_1\}}\}$ and $\{X'_0, X'_1, \cdots X'_{\text{min}\{s_2, r_2\}}\}$ are given by
\begin{enumerate}
\item[(a)] $ \underset{t=0}{\overset{\text{min}\{s_1, r_1\}}{\sum}} \left[ \underset{j=0}{\overset{l}{\sum}} (-1)^j \ {_l}C_j \ {_{(s_1-l)}}C_{(t-j)} \ {_{(r_1-l)}}C_{(t-j)}\right] X_{\text{min}\{s_1, r_1\}-t}$

\NI for all $0 \leq l \leq \text{min}\{s_1, r_1\}$ and $0 \leq r_1 \leq k-s_1-s_2$

\NI where $X_{\text{min}\{s_1, r_1\}-t} = (-1)^t \ t! \ 2^t \ \underset{i=t}{\overset{r_1-1}{\prod}}[x^2 - x-2(s_1+i)], k , s_1$ and $s_2$ are fixed integers

\item[(b)] $ \underset{t=0}{\overset{\text{min}\{s_2, r_2\}}{\sum}} \left[ \underset{j=0}{\overset{m}{\sum}} (-1)^j \ {_l}C_j \ {_{(s_2-l)}}C_{(t-j)} \ {_{(r_2-l)}}C_{(t-j)}\right] X'_{\text{min}\{s_2, r_2\}-t}$

\NI for all $0 \leq l \leq \text{min}\{s_2, r_2\}$ and $0 \leq r_2 \leq k-s_1-s_2$

\NI where $X_{\text{min}\{s_2, r_2\}-t} = (-1)^t \ t! \ \underset{m=i}{\overset{r-1}{\prod}} [x-(s+l)], k, s_1$ and $s_2$ are fixed integers.
\end{enumerate}
\item[(ii)] The set of all distinct eigenvalues of block submatrices $\left(\widetilde{A}_{2r_1+r_2, 2r_1+r_2}\right)_{\substack{0 \leq 2r_1+r_2 \leq 2k-2s_1-2s_2-1 \\ 0 \leq r_1, r_2 \leq k-s_1-s_2-1}}$ of the Gram matrix $\widetilde{G}_{2s_1+s_2}$ of signed partition algebra with entries $\{X_0, X_1, \cdots, X_{\text{min}\{s_1, r_1\}}\}$ and $\{X'_0, X'_1, \cdots X'_{\text{min}\{s_2, r_2\}}\}$ are given by
\begin{enumerate}
\item[(a)] $ \underset{t=0}{\overset{\text{min}\{s_1, r_1\}}{\sum}} \left[ \underset{j=0}{\overset{l}{\sum}} (-1)^j \ {_l}C_j \ {_{(s_1-l)}}C_{(t-j)} \ {_{(r_1-l)}}C_{(t-j)}\right] X_{\text{min}\{s_1, r_1\}-t}$

\NI for all $0 \leq l \leq \text{min}\{s_1, r_1\}$ and $0 \leq r_1 \leq k-s_1-s_2$

\NI where $X_{\text{min}\{s_1, r_1\}-t} = (-1)^t \ t! \ 2^t \ \underset{i=t}{\overset{r_1-1}{\prod}}[x^2 - x-2(s_1+i)], k , s_1$ and $s_2$ are fixed integers

\item[(b)] $ \underset{t=0}{\overset{\text{min}\{s_2, r_2\}}{\sum}} \left[ \underset{j=0}{\overset{m}{\sum}} (-1)^j \ {_l}C_j \ {_{(s_2-l)}}C_{(t-j)} \ {_{(r_2-l)}}C_{(t-j)}\right] X'_{\text{min}\{s_2, r_2\}-t}$

\NI for all $0 \leq l \leq \text{min}\{s_2, r_2\}$ and $0 \leq r_2 \leq k-s_1-s_2-1,$

\NI where $X_{\text{min}\{s_2, r_2\}-t} = (-1)^t \ t! \ \underset{m=i}{\overset{r-1}{\prod}} [x-(s+l)], k, s_1$ and $s_2$ are fixed integers.
\end{enumerate}
\end{enumerate}
\end{thm}

\begin{proof}
Since the $\{e\}$-connected components  cannot be replaced by $\mathbb{Z}_2$-connected components, we can rearrange the diagrams such that the block matrices $A'_{2r_1+r_2, 2r_1+r_2} \left( \widetilde{A}'_{2r_1+r_2, 2r_1+r_2}\right)$ becomes the tensor product of matrices say $A'_{2r_1, 2r_1} \otimes A'_{r_2, r_2} \left(\widetilde{A}'_{2r_1, 2r_1} \otimes \widetilde{A}'_{r_2, r_2} \right)$ respectively.

\NI Let $U^{(d, P)}_{(d, P)} \in \mathbb{J}^{2r_1+r_2}_{2s_1+s_2} \left( U^{(\widetilde{d}, \widetilde{P})}_{(\widetilde{d}, \widetilde{P})} \in \widetilde{\mathbb{J}}^{2r_1+r_2}_{2s_1+s_2}\right)$  such that $\phi((d, P)) = \lambda \left( \widetilde{\phi}((\widetilde{d}, \widetilde{P})) = \lambda\right)$ where $\lambda \in \Omega_{s_1, s_2}^{r_1, r_2}$  is as in Definition \ref{D2.11},  $\mathbb{J}^{2r_1+r_2}_{2s_1+s_2}$ and $\widetilde{\mathbb{J}}^{2r_1+r_2}_{2s_1+s_2}$ are as in Notation \ref{N2.12}.

We shall draw a diagram  using the diagram $d \left( \widetilde{d}) \right)$ with $2s_1+s_2$ through classes respectively as follows.
\begin{enumerate}
  \item[(i)] Draw $d \left( \widetilde{d}\right)$ denoted by $d^+$ in the top row and a copy of $d^+ \left( \widetilde{d}^+\right)$ denoted by $d^- \left( \widetilde{d}^-\right)$ in the bottom row.
  \item[(ii)] Fix the positions of $\mathbb{Z}_2$-connected component and among the $ s_1 +  r_1$ number of pairs $\{e\}$-connected components in the top row choose $s_1$ pairs of $\{e\}$ connected components and join each pair of $\{e\}$-connected component in the top row with the respective pair of $\{e\}$-connected component in the bottom row by vertical edges.

      Similarly, Fix the positions of $\{e\}$-connected components and among the $ s_2 +  r_2$ number of pairs $\mathbb{Z}_2$-connected components in the top row choose $s_2$ number of $\mathbb{Z}_2$-connected components and join each  $\mathbb{Z}_2$-connected component in the top row with the respective $\mathbb{Z}_2$-connected component in the bottom row by vertical edges.
  \item[(iii)] Denote the collection of  diagrams obtained by fixing $\mathbb{Z}_2$-connected components  as $\Omega^{s_1+r_1, s_1}_{d^{\{e\}}} \left( \widetilde{\Omega}^{s_1+r_1, s_1}_{\widetilde{d}^{\{e\}}}\right)$ and denote the collection of  diagrams obtained by fixing $\{e\}$-connected components  as $\Omega^{s_2+r_2, s_2}_{d^{\mathbb{Z}_2}} \\ \left( \widetilde{\Omega}^{s_1+r_1, s_1}_{\widetilde{d}^{\mathbb{Z}_2}}\right).$
\end{enumerate}

\NI It is clear that the number of such diagrams with $2s_1+s_2$ through classes is ${_{(s_1+r_1)}}C_{s_1} \ {_{(s_2+r_2)}}C{s_2}$ and $\mathbb{J}^{2r_1+r_2}_{2s_1+s_2} = \underset{d}{\cup} \ \Omega^{s_1+r_1, s_1}_{d^{\{e\}}} \times \Omega^{s_2+r_2, s_2}_{d^{\mathbb{Z}_2}} \left( \widetilde{\mathbb{J}}^{2r_1+r_2}_{2s_1+s_2} = \underset{\widetilde{\widetilde{d}}}{\cup} \ \widetilde{\Omega}^{s_1+r_1, s_1}_{d^{\{e\}}} \times \widetilde{\Omega}^{s_2+r_2, s_2}_{d^{\mathbb{Z}_2}} \right)$ where $\mathbb{J}^{2r_1+r_2}_{2s_1+s_2}$ and $\widetilde{\mathbb{J}}^{2r_1+r_2}_{2s_1+s_2}$  are as in Notation \ref{N2.12}.


Suppose $U^{(d', P')}_{(d', P')} \in \mathbb{J}^{2r'_1+r'_2}_{2s_1+s_2}$ then the entry corresponding to the product  $U^{(d, P)}_{(d, P)} \cdot U^{(d', P')}_{(d', P')}$ in the Gram matrix is either $ x^{2r''_1+r''_2}$ with $r''_1 < r_1$ and $r''_2 < r_2$ or $ 0$ mod $\lambda.$

\NI \textbf{Case (i):} If the entry corresponding to the product $U^{(d, P)}_{(d, P)} \cdot U^{(d', P')}_{(d', P')}$ is $ x^{2r''_1+r''_2}$ with $r''_1 < r_1$ and $r''_2 < r_2$ then it becomes zero after applying column operations to eliminate the entries corresponding to the diagrams coarser than $U^{(d, P)}_{(d, P)}$ by Lemma 3.20 in \cite{KP}.

\NI \textbf{Case (ii):} If the entry corresponding to the product $U^{(d, P)}_{(d, P)} \cdot U^{(d', P')}_{(d', P')}$ is $ 0$ mod $\lambda$ then it remains zero even after applying the column operations to eliminate the entries corresponding to the diagrams coarser than $U^{(d, P)}_{(d, P)}$  by Lemma 3.19 and Lemma 3.23 in \cite{KP}.

Rearranging the diagrams $U^{(d, P)}_{(d, P)} \in \mathbb{J}_{2s_1+s_2}^{2r_1+r_2}$ in such a way that

\centerline{$A'_{2r_1, 2r_1} = \underset{d}{\prod} A^{s_1+r_1, s_1}_{d}$ and $A'_{r_2, r_2} =  \underset{d}{\prod} A^{s_2+r_2, s_2}_{d}$}
\NI where $A^{s_1+r_1, s_1}_{d}$ and $A^{s_2+r_2, s_2}_{d}$ are symmetric diagram matrices of size ${_{(s_1+r_1)}}C_{s_1}$ and ${_{(s_2+r_2)}}C_{s_2}$ respectively. The  entries of the symmetric diagram matrices $A^{s_1+r_1, s_1}_{d}$ and $A^{s_2+r_2, s_2}_{d}$ are $\{X_0, X_1, \cdots, X_{\text{min}\{s_1, r_1\}}\}$ and $\{X'_0, X'_1, \cdots, X'_{\text{min}\{s_2, r_2\}}\}$ respectively and the diagrams corresponding to the entries of the symmetric diagram matrices $A^{s_1+r_1, s_1}_{d}$ and $A^{s_2+r_2, s_2}_{d}$ belong to $\Omega^{s_1+r_1, s_1}_{d^{\{e\}}}$ and $\Omega^{s_2+r_2, s_2}_{d^{\mathbb{Z}_2}}$ respectively.

Similarly, we can rearrange the diagrams in signed partition algebra in such a way that

\centerline{$\widetilde{A}'_{2r_1, 2r_1} = \underset{\widetilde{d}}{\prod} \widetilde{A}^{s_1+r_1, s_1}_{\widetilde{d}}$ and $\widetilde{A}'{r_2, r_2} =  \underset{\widetilde{d}}{\prod} A^{s_2+r_2, s_2}_{\widetilde{d}}$}
\NI where $\widetilde{A}^{s_1+r_1, s_1}_{\widetilde{d}}$ and $\widetilde{A}^{s_2+r_2, s_2}_{\widetilde{d}}$ are symmetric diagram matrices of size ${_{(s_1+r_1)}}C_{s_1}$ and ${_{(s_2+r_2)}}C_{s_2}$ respectively. The  entries of the symmetric diagram matrices $\widetilde{A}^{s_1+r_1, s_1}_{\widetilde{d}}$ and $\widetilde{A}^{s_2+r_2, s_2}_{\widetilde{d}}$ are $\{X_0, X_1, \cdots, X_{\text{min}\{s_1, r_1\}}\}$ and $\{X'_0, X'_1, \cdots, X'_{\text{min}\{s_2, r_2\}}\}$ respectively and the diagrams corresponding to the entries of the symmetric diagram matrices $\widetilde{A}^{s_1+r_1, s_1}_{\widetilde{d}}$ and $\widetilde{A}^{s_2+r_2, s_2}_{\widetilde{d}}$ belong to $\widetilde{\Omega}^{s_1+r_1, s_1}_{\widetilde{d}^{\{e\}}}$ and $\widetilde{\Omega}^{s_2+r_2, s_2}_{\widetilde{d}^{\mathbb{Z}_2}}$ respectively.

\NI Also, the by Theorem \ref{T2.17}(a) $X_{\text{min}\{s_1, r_1\}-t} = (-1)^t \ t! \ 2^t \ \underset{m=t}{\overset{r_1-1}{\prod}} [x^2-x-2(s_1+m)]$ and $X'_{\text{min}\{s_2, r_2\}-t} = (-1)^t \ t!  \ \underset{m=t}{\overset{r_2-1}{\prod}} [x-(s_2+m)].$

\NI By theorem \ref{T3.9}, the eigenvalues of the symmetric diagram matrix $A^{s_1+r_1, s_1}_{d}\left( \widetilde{A}^{s_1+r_1, s_1}_{d} \right)$ of size ${_{(s_1+r_1)}}C_{s_1}$ are given by

 $ \underset{t=0}{\overset{\text{min}\{s_1, r_1\}}{\sum}} \left[ \underset{j=0}{\overset{l}{\sum}} (-1)^j \ {_l}C_j \ {_{(s_1-l)}}C_{(t-j)} \ {_{(r_1-l)}}C_{(t-j)}\right] X_{\text{min}\{s_1, r_1\}-t}$

\NI for all $0 \leq l \leq \text{min}\{s_1, r_1\}.$

\NI By theorem \ref{T3.9}, the eigenvalues of the symmetric diagram matrix $A^{s_2+r_2, s_2}_{d}\left( \widetilde{A}^{s_2+r_2, s_2}_{d}\right)$ of size ${_{(s_2+r_2)}}C_{s_2}$ are given by

 $ \underset{t=0}{\overset{\text{min}\{s_2, r_2\}}{\sum}} \left[ \underset{j=0}{\overset{m}{\sum}} (-1)^j \ {_l}C_j \ {_{(s_2-l)}}C_{(t-j)} \ {_{(r_2-l)}}C_{(t-j)}\right] X'_{\text{min}\{s_2, r_2\}-t}$

\NI for all $0 \leq l \leq \text{min}\{s_2, r_2\}.$

\NI Thus, the eigenvalues of the block submatrix $A'_{2r_1, 2r_1} \left( \widetilde{A}'_{2r_1, 2r_1}\right)$ which is a direct product of symmetric diagram matrices $A^{s_1+r_1, s_1}_{d} \left( \widetilde{A}^{s_1+r_1, s_1}_{\widetilde{d}}\right)$ are given by

 $ \underset{t=0}{\overset{\text{min}\{s_1, r_1\}}{\sum}} \left[ \underset{j=0}{\overset{l}{\sum}} (-1)^j \ {_l}C_j \ {_{(s_1-l)}}C_{(t-j)} \ {_{(r_1-l)}}C_{(t-j)}\right] X_{\text{min}\{s_1, r_1\}-t}$

\NI for all $0 \leq l \leq \text{min}\{s_1, r_1\}.$

\NI Thus, the eigenvalues of the block submatrix $A'_{r_2, r_2}\left( \widetilde{A}'_{r_2, r_2}\right)$ which is a direct product of symmetric diagram matrices $A^{s_2+r_2, s_2}_{d} \left( \widetilde{A}^{s_2+r_2, s_2}_{d}\right)$ are given by

 $ \underset{t=0}{\overset{\text{min}\{s_2, r_2\}}{\sum}} \left[ \underset{j=0}{\overset{l}{\sum}} (-1)^j \ {_l}C_j \ {_{(s_2-l)}}C_{(t-j)} \ {_{(r_2-l)}}C_{(t-j)}\right] X'_{\text{min}\{s_2, r_2\}-t}$

\NI for all $0 \leq l \leq \text{min}\{s_2, r_2\}.$

\NI Therefore, the eigenvalues of the Gram matrix $G_{2s_1+s_2}$ of the algebra of $\mathbb{Z}_2$-relations are given by

\begin{enumerate}
\item[(a)] $ \underset{t=0}{\overset{\text{min}\{s_1, r_1\}}{\sum}} \left[ \underset{j=0}{\overset{l}{\sum}} (-1)^j \ {_l}C_j \ {_{(s_1-l)}}C_{(t-j)} \ {_{(r_1-l)}}C_{(t-j)}\right] X_{\text{min}\{s_1, r_1\}-t}$

\NI for all $0 \leq l \leq \text{min}\{s_1, r_1\}$ and $0 \leq r_1 \leq k-s_1-s_2$

\NI where $X_{\text{min}\{s_1, r_1\}-t} = (-1)^t \ t! \ 2^t \ \underset{i=t}{\overset{r_1-1}{\prod}}[x^2 - x-2(s_1+i)], k , s_1$ and $s_2$ are fixed integers

\item[(b)] $ \underset{t=0}{\overset{\text{min}\{s_2, r_2\}}{\sum}} \left[ \underset{j=0}{\overset{m}{\sum}} (-1)^j \ {_l}C_j \ {_{(s_2-l)}}C_{(t-j)} \ {_{(r_2-l)}}C_{(t-j)}\right] X'_{\text{min}\{s_2, r_2\}-t}$

\NI for all $0 \leq l \leq \text{min}\{s_2, r_2\}$ and $0 \leq r_2 \leq k-s_1-s_2$

\NI where $X_{\text{min}\{s_2, r_2\}-t} = (-1)^t \ t! \ \underset{m=i}{\overset{r-1}{\prod}} [x-(s+l)], k, s_1$ and $s_2$ are fixed integers.
\end{enumerate}






\end{proof}
\section{\textbf{Appendix}}

To compute the eigenvalues of the symmetric diagram matrix $A^{7, 4}$ of size $35$, we need the eigenvalues of the symmetric diagram matrices $A^{2, 1}$ and $A^{5, 3}$ of sizes $3$ and $10$ respectively. First, we shall compute the eigenvalues of the symmetric diagram matrix $A^{2, 1}$ of size $3.$

\begin{center}
$A^{2, 1} = \left(
              \begin{array}{ccc}
                x_1 & x_0 & x_0 \\
                x_0 & x_1 & x_0 \\
                x_0 & x_0 & x_1 \\
              \end{array}
            \right)
$
\end{center}

\NI \textbf{Step 1:} Apply the following row and column operations on $A^{2,1}:$

\centerline{$R_1 \leftrightarrow R_1 - R_2, C_2 \leftrightarrow C_2 + C_1$ and $R_2 \leftrightarrow R_2 - R_3, C_3 \leftrightarrow C_2 + C_3$}

then the reduced matrix is as follows:

\begin{center}
$\left(
   \begin{array}{ccc}
     x_1-x_0 & 0 & 0 \\
     0 & x_1-x_0 & 0 \\
     x_0 & 2x_0 & x_1+2x_0 \\
   \end{array}
 \right)
$
\end{center}

\NI Thus, the eigenvalues of the symmetric diagram matrix $A^{2, 1}$ of size $3$ are $x_1-x_0, x_1+2x_0.$

\NI Secondly, we shall compute the eigenvalues of the symmetric diagram matrix $A^{5, 3}$ of size $10.$

\begin{center}
$A^{5, 3} = \left(
              \begin{array}{cccccccccc}
                x_2 & x_1 & x_1 & x_1 & x_0 & x_0 & x_1 & x_1 & x_0 & x_1 \\
                x_1 & x_2 & x_1 & x_0 & x_1 & x_0 & x_1 & x_0 & x_1 & x_1 \\
                x_1 & x_1 & x_2 & x_0 & x_0 & x_1 & x_0 & x_1 & x_1 & x_1 \\
                x_1 & x_0 & x_0 & x_2 & x_1 & x_1 & x_1 & x_1 & x_0 & x_1 \\
                x_0 & x_1 & x_0 & x_1 & x_2 & x_1 & x_1 & x_0 & x_1 & x_1 \\
                x_0 & x_0 & x_1 & x_1 & x_1 & x_2 & x_0 & x_1 & x_1 & x_1 \\
                x_1 & x_1 & x_0 & x_1 & x_1 & x_0 & x_2 & x_1 & x_1 & x_0 \\
                x_1 & x_0 & x_1 & x_1 & x_0 & x_1 & x_1 & x_2 & x_1 & x_0 \\
                x_0 & x_1 & x_1 & x_0 & x_1 & x_1 & x_1 & x_1 & x_2 & x_0 \\
                x_1 & x_1 & x_1 & x_1 & x_1 & x_1 & x_0 & x_0 & x_0 & x_2 \\
              \end{array}
            \right)
$
\end{center}

\NI \textbf{Step 1:} apply the following row and column operations on the symmetric diagram matrix $A^{5, 3}.$

\centerline{$R_i \leftrightarrow R_i - R_{3+i}$ and $C_{3+i} \leftrightarrow C_{3+i} + C_i \ \ $ for $1 \leq i \leq 3$}

\NI then the reduced matrix is as follows:

\begin{center}
$\left(
   \begin{array}{cc}
     A_{1, 1} & 0 \\
     \ast & A_{1, 2} \\
   \end{array}
 \right)
$
\end{center}

\NI where $A_{1, 1} = \left(
                        \begin{array}{ccc}
                          y_1 & y_0 & y_0 \\
                          y_0 & y_1 & y_0 \\
                          y_0 & y_0 & y_1 \\
                        \end{array}
                      \right)
$ is a symmetric diagram matrix of size $3$ with $y_1 = x_2 - x_1$ and $y_0 = x_1 - x_0.$ Using induction the eigenvalues of the symmetric diagram matrix $A_{1, 1}$ are $y_1 - y_0$ and $y_1+2y_0.$ i.e., $x_2-2x_1+x_0$ and $x_2+x_1-2x_0$ respectively.

\NI \textbf{Step 2:} Apply the following row and column operations on $A_{1, 2}.$

\centerline{$R_4 \leftrightarrow R_4 - R_6, R_7 \leftrightarrow R_7 - R_9$ and $C_6 \leftrightarrow C_6 + C_4, C_9 \leftrightarrow C_9 + C_7$}

\NI By interchanging the rows and columns of the reduced matrix suitably, we get

\begin{center}
$\left(
  \begin{array}{cc}
    A_{2, 1} & 0 \\
    \ast & A_{2, 2} \\
  \end{array}
\right)$
\end{center}

\NI where $A_{2, 1} = \left(
                        \begin{array}{cc}
                          x_2-x_0 & x_1-x_0 \\
                         2x_1-2x_0  & x_2-x_1 \\
                        \end{array}
                      \right)
$ is a submatrix of the symmetric diagram matrix $A_{1, 1}$ obtained in Step 1. Therefore, the eigenvalues of $A_{2, 1}$ are same as the eigenvalues of $A_{1, 1}.$

\NI \textbf{Step 3:} Apply the following row and column operations on $A_{2, 2}.$

\centerline{$R_5 \leftrightarrow R_5 - R_6, R_8 \leftrightarrow R_8 - R_9$ and $C_6 \leftrightarrow C_6 + C_5, C_9 \leftrightarrow C_9 + C_8$}

\NI By interchanging the rows and columns of the reduced matrix suitably, we get

\begin{center}
$\left(
  \begin{array}{cc}
    A_{3, 1} & 0 \\
    \ast & A_{3, 2} \\
  \end{array}
\right)$
\end{center}

\NI where $A_{3, 2} = \left(
                        \begin{array}{ccc}
                          x_2+3x_1+2x_0 & 2x_1+x_0 & x_1 \\
                          4x_1+2x_0 & x_2+2x_1 & x_0 \\
                          6x_1 & 3x_0 & x_2 \\
                        \end{array}
                      \right)
$, $A_{3, 1} $ is same as $A_{2, 1}$ which is a submatrix of the symmetric diagram matrix $A_{1, 1}$ obtained in Step 1. Therefore, the eigenvalues of $A_{3, 1}$ are same as the eigenvalues of $A_{1, 1}.$

\NI \textbf{Step 4:}Apply the following row and column operation on $A_{3, 2}.$

\centerline{$C_1 \leftrightarrow C_1+C_2+C_3, R_2 \leftrightarrow R_2 - R_1$ and $R_3 \leftrightarrow R_3 - R_1$}

The reduced matrix is as follows:

\begin{center}
$\left(
   \begin{array}{cc}
     A' & \ast \\
     0 & A'' \\
   \end{array}
 \right)
$
\end{center}

\NI where $A'$ is a $1 \times 1$ matrix whose entry is sum of the entries of the symmetric diagram matrix $A^{5, 3}$ i.e., $x_2+6x_1+3x_0$ and $A'' = \left(
                                       \begin{array}{cc}
                                         x_2-x_0 & x_0-x_1 \\
                                        2x_0-2x_1 & x_2-x_1 \\
                                       \end{array}
                                     \right)
$ which is again  a submatrix of the symmetric diagram matrix $A_{1, 1}$ obtained in Step 1. Therefore, the eigenvalues of $A''$ are same as the eigenvalues of $A_{1, 1}.$

\NI Therefore, the eigenvalues of the symmetric diagram matrix $A^{5, 3}$ are $x_2+6x_1+3x_0, x_2-2x_1+x_0$ and $x_2+x_1-2x_0.$

\NI Now, we shall compute the eigenvalues of the symmetric diagram matrix $A^{7, 4}$ using induction.

The following are the diagrams in $\Omega^{s+r, s}$ when $s=4$ and $r = 3:$

\begin{center}
\includegraphics{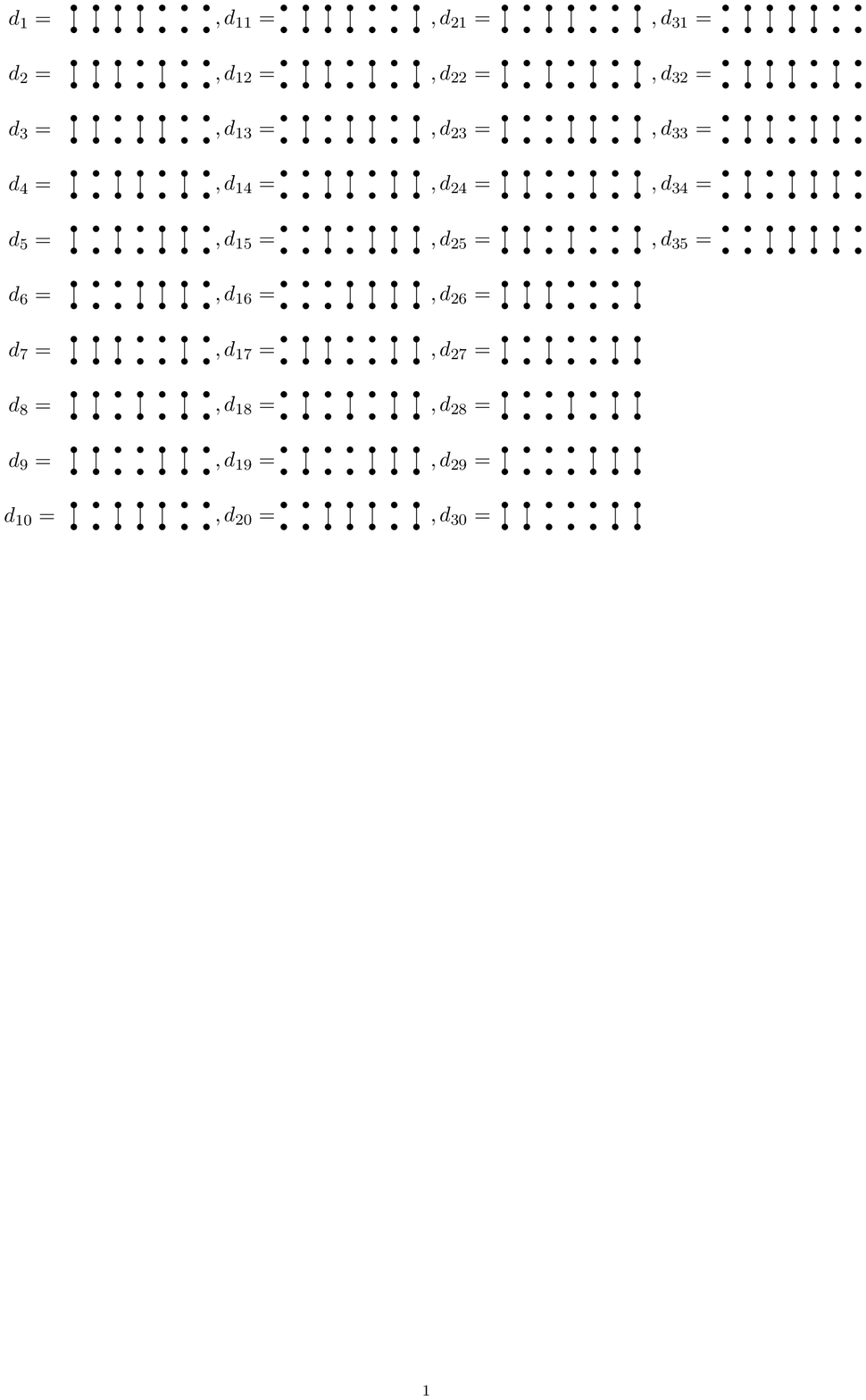}
\end{center}

\NI \textbf{Step 1:} Applying the following row and column operations on the symmetric diagram matrix $A^{4+3, 4}:$

\centerline{$R_{d_i} \leftrightarrow R_{d_i} - R_{d_{10+i}} \ \ \ C_{d_{10+i}} \leftrightarrow C_{d_{10+i}} + C_i \ \ \ 1 \leq i \leq 10$}
\NI where $d_i, 1 \leq i \leq 10 $ belongs to $\Omega_{\{1\}, \{7\}}$ and $d_{10+i}, 1 \leq i \leq 10$ belongs to $\Omega_{\{7\},, \{1\}}.$

The reduced matrix is as follows:

\begin{center}
$\left(
   \begin{array}{cc}
     A_{1, 1} & 0 \\
     \ast & A_{1, 2} \\
   \end{array}
 \right)
$
\end{center}

\NI where $A_{1, 1}$ is a symmetric diagram matrix of size $10$ and

$\hspace{-1cm}A_{1, 1} = \left(
              \begin{array}{cccccccccc}
                x_3 - x_2 & x_2-x_1 & x_2-x_1 & x_2-x_1 & x_1-x_0 & x_1-x_0 & x_2-x_1 & x_2-x_1 & x_1-x_0 & x_2-x_1 \\
                x_2-x_1 & x_3-x_2 & x_2-x_1 & x_1-x_0 & x_2-x_1 & x_1-x_0 & x_2-x_1 & x_1-x_0  & x_2-x_1 & x_2-x_1 \\
                x_2-x_1 & x_2-x_1 & x_3-x_2 & x_1-x_0 & x_1-x_0 & x_2-x_1 & x_1-x_0 & x_2-x_1 & x_2-x_1 & x_2-x_1 \\
                x_2-x_1 & x_1-x_0 & x_1-x_0 & x_3-x_2 & x_2-x_1 & x_2-x_1 & x_2-x_1 & x_2-x_1 & x_1-x_0 & x_2-x_1 \\
                x_1-x_0 & x_2-x_1 & x_1-x_0 & x_2-x_1 & x_3-x_2 & x_2-x_1 & x_2-x_1 & x_1-x_0 & x_2-x_1 & x_2-x_1 \\
                x_1-x_0 & x_1-x_0 & x_2-x_1 & x_2-x_1 & x_2-x_1 & x_3-x_2 & x_1-x_0 & x_2-x_1 & x_2-x_1 & x_2-x_1 \\
                x_2-x_1 & x_2-x_1 & x_1-x_0 & x_2-x_1 & x_2-x_1 & x_1-x_0 & x_3-x_2 & x_2-x_1 & x_2-x_1 & x_1-x_0 \\
                x_2-x_1 & x_1-x_0 & x_2-x_1 & x_2-x_1 & x_1-x_0 & x_2-x_1 & x_2-x_1 & x_3-x_2 & x_2-x_1 & x_1-x_0  \\
                x_1-x_0 & x_2-x_1 & x_2-x_1 & x_1-x_0 & x_2-x_1 & x_2-x_1 & x_2-x_1 & x_2-x_1 & x_3-x_2 & x_1-x_0 \\
                x_2-x_1 & x_2-x_1 & x_2-x_1 & x_2-x_1 & x_2-x_1 & x_2-x_1 & x_1-x_0 & x_1-x_0 & x_1-x_0 & x_3-x_2 \\
              \end{array}
            \right).
$

\NI Using induction, the eigenvalues of the symmetric diagram matrix $A_{1, 1}$ are $x_3+5x_2-3x_1-3x_0, x_3-3x_2+3x_1-x_0$ and $x_3-3x_1+2x_0.$

\NI \textbf{Step 2:} The diagrams corresponding to the entries of the matrix $A_{1, 2}$ belong to $\Omega_{I_i, J_i}$ where $I_1 = \{7\}, I_2 = \{1, 7\}, I_3 = \{ \ \}$ and $J_1 = \{1\}, \{ \ \}, \{1, 7\}$ respectively. Apply the following row and column operations on $A_{1, 2}:$

$R_d \leftrightarrow R_d - R_{d^{\ast}} \ \ \ \ \forall d \in \Omega_{I_i \cup \{2\}, J_i \cup \{6\}}$ and

$C_{d^{\ast}} \leftrightarrow C_{d^{\ast}} + C_d \ \ \ \ \forall d^{\ast} \in \Omega_{I_i \cup \{6\}, J_i \cup \{ 2\}} \ \ \forall 1 \leq i \leq 3.$

\NI By interchanging the rows and columns suitably the reduced matrix is as follows:
\begin{center}
$ \left(
    \begin{array}{cc}
      A_{2, 1} & 0 \\
      \ast &  A_{2, 2}\\
    \end{array}
  \right)
$
\end{center}
\NI where $A_{2, 1} = \left(
              \begin{array}{ccccccc}
                x_3-x_1 & x_2-x_0 & x_2-x_0 & x_1-x_0 & x_2-x_1 & x_2-x_1 & x_2-x_1 \\
                x_2-x_0 & x_3-x_1 & x_2-x_0 & x_2-x_1 & x_1-x_0 & x_2-x_1 & x_2-x_1 \\
                x_2-x_0 & x_2-x_0 & x_3-x_1 & x_2-x_1 & x_2-x_1 & x_1-x_0 & x_2-x_1 \\
                2x_1-2x_0 & 2x_2-2x_1 & 2x_2-2x_1 & x_3-x_2 & x_2-x_1 & x_2-x_1 & x_1-x_0 \\
                2x_2-2x_1 & 2x_1-2x_0 & 2x_2-2x_1 & x_2-x_1 & x_3-x_2 & x_2-x_1 & x_1-x_0 \\
                2x_2-2x_1 & 2x_2-2x_1 & 2x_1-2x_0 & x_2-x_1 & x_2-x_1 & x_3-x_2 & x_1-x_0 \\
                2x_2-2x_1 & 2x_2-2x_1 & 2x_2-2x_1 & x_1-x_0 & x_1-x_0 & x_1-x_0 & x_3-x_2 \\
              \end{array}
            \right)
$ is same as $B_{1, 2}$ a submatrix of $A_{1, 1}$ obtained after applying the row and column operations.Thus, the eigenvalues of the matrix $A_{2, 1}$ are same as the eigenvalues of matrix $A_{1, 1}.$

\NI \textbf{Step 3:} The diagrams corresponding to the entries of the matrix $A_{2,2}$ belong to $\Omega_{I_i, J_i}$ where $I_1 = \{6, 7\}, I_2 = \{1, 6, 7\}, I_3 = \{ 6 \}, I_4 =\{2, 6, 7\}, I_5 = \{1, 2, 6, 7\}, I_6 = \{2, 6\}, I_7 = \{7\}, I_8= \{1, 7\}, I_9 = \{ \ \}$ and $J_1 = \{1, 2\}, J_2 = \{ 2 \}, J_3 = \{1, 2, 7\}, J_4=\{1\}, J_5 = \{ \ \}, J_6=\{1, 7\}, J_7 = \{1, 2, 6\}, J_8 = \{2,6\}, J_9 = \{1,2,6,7\}$ respectively. Apply the following row and column operations on $A_{2, 1}:$

$R_d \leftrightarrow R_d - R_{d^{\ast}} \ \ \ \ \forall d \in \Omega_{I_i \cup \{3\}, J_i \cup \{5\}}$ and

$C_{d^{\ast}} \leftrightarrow C_{d^{\ast}} + C_d \ \ \ \ \forall d^{\ast} \in \Omega_{I_i \cup \{5\}, J_i \cup \{ 3\}} \ \ \forall 1 \leq i \leq 3^2.$

By interchanging the rows and columns suitably the reduced matrix is as follows:

\begin{center}
$ \left(
    \begin{array}{cc}
      A_{3, 1} & 0 \\
      \ast &  A_{3, 2}\\
    \end{array}
  \right)
$
\end{center}
\NI where $A_{3, 1} = \left(
              \begin{array}{ccccc}
                x_3+x_2-x_1-x_0 & x_2-x_0 & x_2-x_1 & x_2-x_0 & x_2-x_1 \\
                2x_2-2x_0 & x_3-x_1 & x_1-x_0 & 2x_2-2x_1 & x_2-x_1 \\
                4x_2-4x_1 & 2x_1-2x_0 & x_3-x_2 & 2x_2-2x_1 & x_1-x_0 \\
                2x_2-2x_0 & 2x_2-2x_1 & x_2-x_1 & x_3-x_1 & x_1-x_0 \\
                4x_2-4x_1 & 2x_2-2x_1 & x_1-x_0 & 2x_1-2x_0 & x_3-x_2 \\
              \end{array}
            \right)
$ is same as $B_{2, 2}$ obtained from $B_{1, 2}$ after applying row and column operations.Thus, the eigenvalues of the matrix $A_{3, 1}$ are same as the eigenvalues of matrix $A_{1, 1}.$

\NI \textbf{Step 4:} The diagrams corresponding to the entries of the matrix $A_{3,2}$ belong to $\Omega_{I_i, J_i}$ where $1 \leq i \leq 3^3.$ Apply the following row and column operations on $A_{3, 1}:$

$R_d \leftrightarrow R_d - R_{d^{\ast}} \ \ \ \ \forall d \in \Omega_{I_i \cup \{3\}, J_i \cup \{4\}}$ and $\forall d \in \Omega_{I_i \cup \{4\}, J_i \cup \{5\}}$

$C_{d^{\ast}} \leftrightarrow C_{d^{\ast}} + C_d \ \ \ \ \forall d \in \Omega_{I_i \cup \{4\}, J_i \cup \{3\}}$ and $\forall d \in \Omega_{I_i \cup \{5\}, J_i \cup \{4\}} \ \ \forall 1 \leq i \leq 3^3.$

By interchanging the rows and columns suitably the reduced matrix is as follows:

\begin{center}
$ \left(
    \begin{array}{cc}
      A_{4, 1} & 0 \\
      \ast &  A_{4, 2}\\
    \end{array}
  \right)
$
\end{center}
\NI where $A_{4, 1}$ is same as $A_{3, 1}.$ Thus, the eigenvalues of the matrix $A_{4, 1}$ are same as the eigenvalues of matrix $A_{1, 1}.$

\NI \textbf{Step 5:} Apply the following row and column operations on $A_{4, 2}:$

\begin{enumerate}
  \item[(i)] Fix a $d \in \Omega_{\{7\}, \{1\}} \setminus \underset{i=1}{\overset{3^4}{\cup}} \Omega_{I_i, J_i}.$

      \centerline{$C_d \leftrightarrow C_d + \underset{d'}{\sum} C_{d'}$}

  where $d' \in \Omega_{\{7\}, \{1\}} \setminus \underset{i=1}{\overset{3^4}{\cup}} \Omega_{I_i, J_i}.$

  \centerline{$R_{d'} \leftrightarrow R_{d'} - R_d$ for all $d' \in \Omega_{\{7\}, \{1\}} \setminus \underset{i=1}{\overset{3^4}{\cup}} \Omega_{I_i, J_i}.$}

  \item[(ii)] Fix a $d \in \Omega_{\{1, 7\}, \{ \ \}} \setminus \underset{i=1}{\overset{3^4}{\cup}} \Omega_{I_i, J_i}.$

      \centerline{$C_d \leftrightarrow C_d + \underset{d'}{\sum} C_{d'}$}

  where $d' \in \Omega_{\{1,7\}, \{ \ \}} \setminus \underset{i=1}{\overset{3^4}{\cup}} \Omega_{I_i, J_i}.$

  \centerline{$R_{d'} \leftrightarrow R_{d'} - R_d$ for all $d' \in \Omega_{\{1, 7\}, \{ \ \}} \setminus \underset{i=1}{\overset{3^4}{\cup}} \Omega_{I_i, J_i}.$}

  \item[(iii)]   Fix a $d \in \Omega_{\{ \ \}, \{1, 7\}} \setminus \underset{i=1}{\overset{3^4}{\cup}} \Omega_{I_i, J_i}.$

      \centerline{$C_d \leftrightarrow C_d + \underset{d'}{\sum} C_{d'}$}

  where $d' \in \Omega_{\{ \ \}, \{1, 7\}} \setminus \underset{i=1}{\overset{3^4}{\cup}} \Omega_{I_i, J_i}.$

  \centerline{$R_{d'} \leftrightarrow R_{d'} - R_d$ for all $d' \in \Omega_{\{ \ \}, \{1, 7\}} \setminus \underset{i=1}{\overset{3^4}{\cup}} \Omega_{I_i, J_i}.$}

\end{enumerate}

interchanging the rows and columns suitably, the reduced matrix is as follows:

\begin{center}
$ \left(
    \begin{array}{cc}
      A_{5, 1} & 0 \\
      \ast &  A_{5, 2}\\
    \end{array}
  \right)
$
\end{center}
\NI where

$A_{5, 1} = \left(
              \begin{array}{ccc}
                x_3+7x_2+9x_1+3x_0 & 3x_2+6x_1+x_0 & 2x_2+3x_1 \\
                6x_2+12x_1+2x_0 & x_3+6x_2+3x_1 & 3x_1+2x_0 \\
                8x_2+12x_1 & 6x_1+4x_0 & x_3+4x_2 \\
              \end{array}
            \right)
$ and $A_{5, 2}$ is same as $A_{3, 1}.$ Thus, the eigenvalues of the matrix $A_{5, 2}$ are same as the eigenvalues of matrix $A_{1, 1}.$

\NI \textbf{Step 6:} apply the following row and column operations on $A_{5, 1}:$

\centerline{$C_1 \leftrightarrow C_1 + C_2 + C_3, R_2 \leftrightarrow R_2 - R_1$ and $R_3 \leftrightarrow R_3 - R_1$}

then the reduced matrix is as follows:

\begin{center}
$\left(
   \begin{array}{cc}
     A' & \ast \\
     0 & A'' \\
   \end{array}
 \right)
$
\end{center}

\NI where $A'$ is a $1 \times 1$ matrix whose entry is sum of the entries of the symmetric diagram matrix $A^{7, 4}$ i.e., $x_3 + 12x_2+ 18x_1 + 4x_0$ and $A'' = \left(
                                                   \begin{array}{cc}
                                                     x_3+3x_2-3x_1-x_0 & 2x_0-2x_2 \\
                                                     3x_0-3x_2 & x_3+2x_2-3x_1 \\
                                                   \end{array}
                                                 \right).
$
We perform the following row and column operation on $A'':$

\centerline{$R_1 \leftrightarrow R_1 + R_2$ and $C_2 \leftrightarrow C_2 - C_1$}

then the reduced matrix is as follows:

\begin{center}
$\left(
   \begin{array}{cc}
     x_3-3x_1+2x_0 & 0 \\
     3x_0-3x_2 & x_3 +5x_2-3x_1-3x_0 \\
   \end{array}
 \right).
$
\end{center}

\NI Thus, the eigenvalues of the matrix $A''$ are $x_3-3x_1+2x_0$ and $x_3+5x_2-3x_1-3x_0.$

\NI Therefore, the eigenvalues of the symmetric diagram matrix $A^{7, 4}$ are $x_3 + 12x_2+ 18x_1 + 4x_0, x_3+5x_2-3x_1-3x_0, x_3-3x_2+3x_1-x_0$ and $x_3-3x_1+2x_0.$

\end{document}